\documentclass[12pt,a4paper]{amsart}
\usepackage[T1]{fontenc}
\usepackage[utf8]{inputenc}
\usepackage{amssymb, amsthm, latexsym}
\usepackage[margin=3cm]{geometry}
\usepackage{xcolor}

\theoremstyle{plain}
\newtheorem{thm}{Theorem}
\newtheorem{lemma}[thm]{Lemma}
\newtheorem{corollary}[thm]{Corollary}
\newtheorem{proposition}[thm]{Proposition}
\newtheorem{fact}[thm]{Fact}
\newtheorem{claim}{Claim}

\theoremstyle{definition}
\newtheorem{definition}[thm]{Definition}
\newtheorem{remark}[thm]{Remark}
\newtheorem{example}[thm]{Example}
\newtheorem{notation}[thm]{Notation}
\newtheorem{question}{Open question}

\numberwithin{thm}{section}
\numberwithin{claim}{thm}

\renewcommand{\a}{\alpha}
\renewcommand{\b}{\beta}
\newcommand{\e}{\varepsilon}
\renewcommand{\d}{\delta}

\renewcommand{\l}{\lambda}

\renewcommand{\o}{\omega}
\newcommand{\s}{\sigma}
\newcommand{\ol}{\overline}
\newcommand{\p}{\pi}

\newcommand{\raj}{\restriction}

\newcommand{\C}{\mathbb{C}}
\newcommand{\R}{\mathbb{R}}
\renewcommand{\empty}{\emptyset}
\newcommand{\abs}[1]{\lvert#1\rvert}
\newcommand{\norm}[1]{\lVert#1\rVert}
\newcommand{\U}{\mathcal{U}}
\newcommand{\ket}[1]{\lvert#1\rangle}
\newcommand{\bra}[1]{\langle#1\rvert}

\begin{document}

\title{On ultraproducts, the spectral theorem and rigged Hilbert spaces}
\date{\today}
\author{{\AA}sa Hirvonen}
\address{Department of Mathematics and Statistics\\
University of Helsinki\\
P.O.Box 68\\
00014 University of Helsinki\\
Finland}
\email{asa.hirvonen@helsinki.fi}
\author{Tapani Hyttinen}
\address{Department of Mathematics and Statistics\\
University of Helsinki\\
P.O.Box 68\\
00014 University of Helsinki\\
Finland}
\email{tapani.hyttinen@helsinki.fi}
\subjclass[2010]{}
\begin{abstract}
We start by showing how to approximate unitary and bounded self-adjoint
operators by operators in finite dimensional spaces. Using ultraproducts we give a precise meaning
for the approximation. In this process we see how the spectral measure is obtained
as an ultralimit of counting measures that arise naturally from the finite dimensional
approximations.  Then we see how generalized distributions can be interpreted  in
the ultraproduct. Finally we study
how one can calculate kernels of operators $K$ by calculating them in the finite dimensional approximations
and how one needs to interpret Dirac  deltas in the ultraproduct in order to
get the kernels as propagators  $\langle x_{1}|K|x_{0}\rangle$.
\end{abstract}

\maketitle

This paper continues a study of finite dimensional approximations of operators on Hilbert spaces. We started the work in \cite{HH}, motivated by a quest for mathematically justified eigenvectors in quantum mechanics. There we built eigenvectors directly as ultraproducts of eigenvectors in the small spaces. This time we look at another way of finding finite dimensional approximations, and reconstruct the spectral decomposition theorem using ultraproducts. We also find structure resembling Rigged Hilbert Spaces, thus providing generalized eigenvectors and bringing us back to our earlier motivation.

In quantum mechanics the states of a system are usually represented by unit vectors in a Hilbert space. Observables of the system are modeled by self-adjoint (often unbounded) operators whose spectrum gives the possible measurement outcomes. If the observable has a finite spectrum and a finite number of eigenvectors, this model coincides with Dirac's bra- and ket-vector representation \cite{dirac}, but in general Dirac's bras and kets don't form a Hilbert space. A key ingredient in Dirac's approach, the delta-functions, were later mathematically justified by Schwartz \cite{schwartz} as distributions, functionals over a set of test functions. Gelfand and Vilenkin found a natural space for these, namely the Rigged Hilbert Spaces. These are triplets of spaces: in addition to a Hilbert space, one considers a dense subspace of it - the test functions - and the space of either linear or antilinear functions over these test functions. This last space, in which the Hilbert space embeds naturally, is the space of distributions. The linear functions correspond to bra vectors, the antilinear functions to ket vectors.

One key question about quantum systems is how they evolve over time. Depending on the approach tot the system, this is described either by the kernel of the time evolution operator, or via the Feynman propagator. The kernel is a function $K(x,y)$ such that the action of the time evolution operator $K$ on a wave function $\varphi$ can be described by
$$
K(\varphi)(y)=\int_S K(x,y)\varphi(x)d\mu(x),
$$
where $S$ is the spectrum of the observable.

We study a method of calculating the kernel via the finite dimensional approximations. Although our method builds on ultraproducts, the space studied is not the full ultraproduct, but a subspace of it. As this is not elementary, we cannot use \L os's theorem, but this is not a problem: as we are only interested in equations, it is enough to have just an embedding of the original space into the ultraproduct space. This also means that our approach does not restrict to only pseudocompact spaces (the metric analogue of pseudofinite spaces, see \cite{GL}).

Discrete approximations of quantum systems have been used and studied
extensively in physics. From the list of references of \cite{HH} one can find examples of this.
They have been studied also in mathematical physics, see e.g. \cite{Ba}.
We got inspired by the work of B. Zilber \cite{Zi}, and set out to compute concrete examples of kernels in \cite{HH}. 
The question we studied there was whether it is possible to calculate the kernels  of
the time evolution operators by calculating them in some kind of finitely generated
approximations of quantum systems which in \cite{HH} were the free particle and
the harmonic oscillator. It turned out that in these cases it is possible.
One can choose the approximations to be simply finite dimensional Hilbert spaces
with rather straight forward approximations of the position operator and the
momentum operator. In the choice of the  approximations of the
time evolution operator one needed to be more creative in the case of
the harmonic oscillator.
To determine in which sense the approximations
approximate the quantum systems was more tricky. For that we used
a rather  heavily modified version of the metric ultraproduct of the
finite dimensional Hilbert spaces. Once we knew how the approximation works, it was
rather straight forward to do the calculations with the help of number theory.

However, following Zilber, we would have liked to be able to
calculate the kernels  by using Dirac deltas in the propagator
style $\langle x_{1}|K|x_{0}\rangle$. In \cite{HH}
our interpretation of Dirac deltas was
as ultraproducts of eigenvectors. In hind sight that was too naive an approach, and didn't work out. The  propagators calculated in the finite dimensional models did not give a correct propagator in the ultraproduct, and not even some kind of ad hoc renormalization would have made the values correct, as there were also divisibility-related discretizing effects stemming from the finite approximations. The remedy in \cite{HH} was to calculate the kernel instead, as this could be done using the propagators in the finite dimensional models, and the method 'averaged out' the discretization effect.

In this paper we continue the work we started in \cite{HH}, but we look at a general case instead of concrete examples. Also, the method is different. 
The first question we look at is the following: Given a separable Hilbert space $H$
and an operator $A$ on it, can one find finite dimensional Hilbert spaces $H_{N}$
and operators $A_{N}$ on them so that
the operators approximate $A$ in the sense that
$H$ with $A$
is isometrically  isomorphic to a submodel  of the  metric ultraproduct
of the spaces $H_{N}$ with operators $A_{N}$? The first problem in finding the pairs $(H_{N},A_{N})$
is that in the case where the operator $A$ is unbounded,
the metric ultraproduct $A^{m}$ of the operators $A_{N}$ can not be well-defined
everywhere in the metric ultraproduct $H^{m}$ of the spaces  $H_{N}$.
In \cite{HH} looking at the position and momentum operators,  we were able to show
that $A^{m}$  is well defined in a suitable part of $H^{m}$ so that the isomorphism
can indeed be found. In this paper we look only at bounded operators and thus this problem does not arise. Although this does restrict the direct applicability of our method, we hope the method can be extended, more on this below.

In the first two sections we consider the approximation question in the case where
$A$ is such that $AA^{*}=A^{*}A=rI$ where $r$ is a positive real  and
$I$ is the identity operator. So, e.g., $A$ can be a unitary operator.  We show how the operators $A_{N}$
can be found, show how the spectral measure $\mu$ for $A$ is obtained as an ultralimit of
counting measures that arise from the  approximations  naturally, and
how the approximations give $A$  as a multiplication operator in the
space $L_{2}(\s (A),\mu )$.

In the third section we study the same questions for bounded self-adjoint operators $A$
and get the same results by reducing the questions to the previous case by looking at
the operator $e^{irA}$ for $r$ a small enough real.

In the fourth section we look at generalized distributions in the context of
the space $L_{2}(\s (A),\mu )$
from the previous sections. See  \cite{RS} for the classical theory of generalized distributions.
Now the generalized distributions (excluding some)
cannot be seen as elements of $H^{m}$. Thus we look at the space
$H^{\infty}$ which is obtained from the classical ultraproduct of the spaces $H_{N}$ by fractioning out
the equivalence relation of being infinitely close to each other in a
metric that arises from
the norms in $H_N$.
This space contains
infinite vectors giving us a chance. The problem here is that the ultraproduct of
the inner products in the spaces $H_{N}$ is not well-defined everywhere in $H^{\infty}\times H^{\infty}$.
It turns out that it is well defined in a large enough subset so that one can interpret
generalized distributions $\theta$ as vectors $u(\theta )$ in $H^{\infty}$ so that
for all continuous $f:\s (A)\rightarrow \C$, where $\C$ is the field of complex numbers,
$\theta(f)=\langle F^m(f)|u(\theta )\rangle$ where $F^m$ is a natural embedding we get from
the isometric embedding of $L_{2}(\s (A),\mu  )$ into $H^{m}$, see above
(the embedding is not into $H^{\infty}$, for details see Section 4).

In the fifth section  we look at ways of  calculating the kernel of an operator
$B$ on $L_{2}(\s (A),\mu )$  (or on $H$). We start by showing that the method from \cite{HH}
works also here assuming that we can find reasonable approximations $B_{N}$ for $B$
in the spaces $H_{N}$. In \cite{HH} we showed how to find these for the time evolution operators
$B$ of the free particle and the harmonic oscillator (there $A$ was the position operator).
Here we show that if $B$ has a kernel, then the reasonable approximations $B_{N}$
can  always be found. However, unlike in the special cases studied in \cite{HH}, here
our proof is essentially existential  and does not give a practical way of finding
the approximations. One must keep in mind that finding the kernels is difficult and
thus one can not expect to have a simple trick that gives them.

In the fifth section we also show that the kernel can be obtained as  the propagator
$\langle x_{1}|B|x_{0}\rangle$
(using the discrete approximations we  can extend $B$ to an operator that
acts also on Dirac deltas)
if one is very careful in choosing the interpretations of
the Dirac deltas  $|x_{1}\rangle$ and $|x_{0}\rangle$ and here, of course, we think of Dirac deltas as
generalized distributions as is common in mathematics. The need to be very careful
comes from the fact that $\langle x_{1}|B|x_{0}\rangle$ is very sensitive to the choice of eigenvectors, as opposed to $\langle F^m(f)|u(\theta )\rangle$. In \cite{HH} we showed that a straightforward approach to eigenvectors can give completely wrong values.
As the result, this Dirac delta method is probably not very practical
if applied in a straightforward manner. From \cite{Ma} one can find an example of a rigorous use of
Dirac deltas in the context where Dirac deltas are interpreted as generalized
distributions.

\begin{question} Can one find approximations $A_{N}$ for
  bounded normal operators $A$ or unbounded self adjoint operators?
  \end{question}

Although it is not immediate, we believe that if one can find finite dimensional approximations of bounded normal operators, one can combine the classical trick of looking at the operators $(A\pm iI)^{-1}$ and our technique for building unbounded operators in ultraproducts from \cite{HH} to find approximations for unbounded self-adjoint operators.
However, even in the case of bounded normal operators we have a serious problem:
How to find the approximations $A_{N}$ so that that they are roughly like
how we chose them in Section 1 and still normal
(with the assumption $AA^{*}=A^{*}A=rI$, there are no problems in
guaranteeing  normality, in fact, this assumption was designed to
get normal approximations). The conventional trick of decomposing normal operators into two self-adjoint ones does not seem to work here, as approximating them in the spaces $H_N$ seems to destroy commutativity, so that we cannot guarantee our approximations to be normal. But without normality,
nothing that we do in this paper works.

\section{The construction} \label{sec:1}

This section presents the basic construction of finite-dimensional approximations that we will use and modify in later sections. Here we study a separable Hilbert space with a (scaled) unitary operator having a cyclic vector. We use the cyclic vector and polynomials to build finite-dimensional approximations such that the original Hilbert space embeds isometrically into their metric ultraproduct. We also get a vector space homomorphism from the space of polynomials (over a compact set $S\subset\C$) into the metric ultraproduct. In section \ref{sec:2} we will construct a suitable measure to make the homomorphism an isometry.

The construction will be used to find similar approximations for the case where the operator is bounded and self-adjoint in section \ref{sec:3}. In section \ref{sec:4} we will study other norms for the ultraproduct space that we will be essential in studying distributions in sections \ref{sec:4} and \ref{sec:5}. In section \ref{sec:non-cyclic} we will have a short look at what one can do if the operator doesn't have a cyclic vector.

So we let $H$ be a separable complex Hilbert space and $A$ an operator on $H$ of the form $qU$, where $q$ is a nonzero complex number and $U$ is unitary. Then $A$ is normal and bounded and $A^{*}\circ A=A\circ A^{*}=r_{A}I$ where $r_{A}$ is a positive real, $I$
is the identity operator, and $A^{*}$ is the adjoint of $A$. This is enough for the constructions we do in this paper (we will use the construction when $A$ is unitary), but we use notation allowing for more general operators, as there are also other operators allowing for this sort of approximations (see the Example at the end of this section). 

From a polynomial $P(X,Y)\in\C [X,Y]$ over the complex numbers $\C$ we get an operator
$P(A,A^{*})$ on $H$ the natural way, e.g. $X^{2}Y(A,A^{*})=A\circ A\circ A^{*}$,
and we say that $\phi\in H$ is \emph{cyclic} if  the set
$\{ P(A,A^{*})(\phi )\vert\ P\in\C [X,Y]\}$ is dense in $H$.
We will here study the case where there is a cyclic vector $\phi$ in $H$
and we pick $\phi$ so that in addition its norm is $1$. Notice that
$H$ can always be split into countably many complete subspaces so that
they are orthogonal to each other, closed under $A$ and $A^{*}$
and each of them has a  cyclic vector, and in section \ref{sec:non-cyclic} we will see how to combine the constructions from this decomposition.

For all $N\in\o$, let $H_{N}$ be the subspace of
$H$ generated by
\begin{equation} 
\{A^{i}(A^{*})^{j}(\phi )\vert\ i\le N,\ j\le N\} .
\end{equation}
In particular, note that $H_0$ is the space generated by $\phi$, and $H_N\subset H_{N+1}$.
Notice that as finite dimensional spaces these are complete spaces.
We write
$H^{-}_{N}$ for the subspace of
$H$ generated by
\begin{equation} 
\{A^{i}(A^{*})^{j}(\phi )\vert\ i< N,\ j\le N\}
\end{equation}
and $H^{+}_{N}$ for the subspace of
$H$ generated by
\begin{equation} 
\{A^{i}(A^{*})^{j}(\phi )\vert\ 0<i\le N,\ j\le N\} .
\end{equation}
As $AA^*=A^*A=r_AI$, the operators $A$ and $A^*$ swap these spaces: $A$  maps  $H^{-}_{N}$ onto $H^{+}_{N}$ and 
$A^{*}$ maps $H^{+}_{N}$ onto $H^{-}_{N}$. To get approximating operators on $H_N$ we can modify $A$ and $A^*$ by a wraparound trick: Let
$W^{-}$ be the  orthogonal complement of $H^{-}_{N}$
in $H_{N}$ and $W^{+}$ be the  orthogonal complement of $H^{+}_{N}$
in $H_{N}$. Then $W^{+}$ and $W^{-}$ have the same dimension and we
let $U$ be a unitary operator from $W^{-}$ onto $W^{+}$.
Now we define an operator $A_{N}$ on  $H_{N}$ as follows.
If $u\in H^{-}_{N}$, then $A_{N}(u)=A(u)$ and if
$u\in W^{-}$, then $A_{N}(u)=\sqrt{r_{A}}U(u)$
(by $\sqrt{r_{A}}$ we mean the one that is positive).
We get the adjoint of $A_{N}$  as follows:  If
$u\in H^{+}_{N}$, then we let $A^{*}_{N}(u)=A^{*}(u)$
and if $u\in W^{+}$, then $A^{*}_{N}(u)=\sqrt{r_{A}}U^{-1}(u)$.
It is easy to check that
$A^{*}_{N}$ is indeed the adjoint of $A_{N}$ and that
$A^{*}_{N}\circ A_{N}=A_{N}\circ A^{*}_{N}=r_{A}I$ and so
$A_{N}$ is normal.
Notice that because of how $A$ and $A^*$ cancel out (up to multiplication by a constant),
$H^{+}_{N}$ is also generated by the set
$\{A^{i}(A^{*})^{j}(\phi )\vert\ i\le N,\ j< N\}$. Thus $H_N^-$ and $H_N^+$ are the subspaces on which $A_N$ and $A_N^*$ respectively behave 'correctly' and
\begin{equation}\label{star1}
A_{N}(A^{i}(A^{*})^{j}(\phi))=A(A^{i}(A^{*})^{j}(\phi ))=A^{i+1}(A^{*})^{j}(\phi ) \textrm{ if } i<N,\,j\leq N
\end{equation}
and
\begin{equation}\label{star2}
A^{*}_{N}(A^{i}(A^{*})^{j}(\phi ))=A^{i}(A^{*})^{j+1}(\phi )\textrm{ if }j<N,\, i\leq N.
\end{equation}

In $H_{N}$ we can also define operators $P(A_{N},A_{N}^{*})$ for polynomials $P(X,Y)$ in the natural way.

For all $N\in\o$, we let $D_{N}$ be the dimension of  $H_{N}$
and we choose eigenvectors $u_{N}(n)$, $n<D_{N}$, of $A_{N}$
with eigenvalues $\l_{N}(n)$.
Notice that
\begin{equation}\label{2stars}
\textrm{$u_{N}(n)$ is also an eigenvector of $A^{*}_N$
with eigenvalue $\overline{\l_{N}(n)}$,}
\end{equation}
\noindent
where $\overline{\l_{N}(n)}$ is the complex conjugate of $\l_{N}(n)$. 
Since $A_{N}$ is normal, 
we can choose these so that they
form an orthonormal basis of $H_{N}$.
Denote $\xi_{N}(n)=\langle u_{N}(n) |\phi \rangle$,
i.e. 
\begin{equation}\label{phi}
\phi =\sum_{n=0}^{D_{N}-1}\xi_{N}(n)u_{N}(n),
\end{equation}
and we choose the vectors $u_{N}(n)$ so that in addition
$\xi_{N}(n)$ is a non-negative real number. One should note that even though the spaces $H_N$ form an increasing sequence, the bases do not extend each other, as the operators $A_N$ and $A_N^*$ do not extend each other (because of the wraparound modification).

Next we want to tie a spectral decomposition of $A$ to our finite-dimensional approximation. So we choose a  natural number $M$ so that it is strictly greater than both the operator norm of $A$
and $\sqrt{r_{A}}$,
and let 
\begin{equation}
S=\{\l\in\C\vert\ -M\le Re(\l )\le M, -M\le Im(\l )\le M\}.
\end{equation}
Notice that there is a real $\e >0$ such that
for all $N<\o$ and $n<D_{N}$,
$\vert\l_{N}(n)\vert <M-\e$ since the norm of
$A_{N}$ is at most the maximum of the norm of $A$
and $\sqrt{r_{A}}$.
For all $X\subseteq\C$,
let $C(X)$ be the vector space of all bounded continuous functions
from $X$  to $\C$. Of these our main interest is in $C(S)$. Notice that
since $S$ is compact, every $f\in C(S)$ is uniformly continuous.
We let
$D(S)$ be the subspace of $C(S)$ that consists
of all functions  $f_{P}(\l )=P(\l ,\ol\l )$, where $P\in\C[X,Y]$
and $\ol\l$ is the complex conjugate of $\l$.
On $D(S)$ we can again define operators $A_{D}$ and $A^{*}_{D}$
in a natural way
\begin{equation}\label{AD}
A_D(f_P)=f_{XP}\quad\text{and}\quad A^*_D(f_P)=f_{YP}
\end{equation}
(e.g. $A_{D}(f_{X^{i}Y^{j}})=f_{X^{i+1}Y^{j}}$).
Notice that $A_{D}(f_{P})(\l )=\l f_{P}(\l)$  and
$A^{*}_{D}(f_{P})(\l )=\ol\l f_{P}(\l)$.

Let's define some mappings that will be crucial in this paper: 

\begin{definition}\label{def:GjaFN}
For $P\in \C[X,Y]$, let
\begin{equation}
G(f_{P})=P(A,A^{*})(\phi )
\end{equation}
and note that it is a homomorphism from $D(S)$ to $H$ (as vector spaces), and that its image is dense in $H$ (by the cyclicity assumption).
Further let
\begin{equation}
F_{N}(f_{P})=\sum_{n<D_{N}}\xi_{N}(n)f_{P}(\l_{N}(n))u_{N}(n),
\end{equation}
and note that this is a homomorphism
from $D(S)$ to $H_{N}$. Notice that $F_{N}(f_{P})=P(A_{N},A_{N}^{*})(\phi )$.
\end{definition}

Now the inner product in $H_{N}$ (inherited from $H$), gives us one norm $\Vert \cdot\Vert_{2}$ on $H_{N}$. Using this, there is a natural way of defining metric ultraproducts of the Hilbert spaces $H_N$. However, we want to consider also other norms in the ultaproduct, so we take as our starting point a classical ultraproduct of the structures $H_N$ and 'dig out' various metric ultraproduct spaces from it. We look at the $\norm{\cdot}_2$-norm here, and consider the other norms in section \ref{sec:4}.

Let $\U$ be an ultrafilter on $\o$
such that (as in \cite{HH}) for all $m\in \o-\{ 0\}$,
$$\{ N\in\o -\{ 0\}\vert\ \sqrt{N}\in\o ,\ m|\sqrt{N}\}\in \U.$$
Let $H^{u}$ be the (classical) ultraproduct $\Pi_{N\in\o}H_{N}/\U$
and $\C^{u}$ be the ultrapower $\C^{\o}/\U$ of the
field of complex numbers $\C$. Then $H^{u}$ is a  vector space over
$\C^{u}$ and by identifying elements of $\C$
with their images  under the canonical embedding
of $\C$ into $\C^{u}$, also a vector space over
$\C$. There is also a pairing $\langle \cdot|\cdot\rangle^{u}$ from
$H^{u}$ to $\C^{u}$ which is obtained as the
ultraproduct of the inner products $\langle\cdot|\cdot\rangle_{N}$
of $H_{N}$. And similarly we get a unary function
$\Vert\cdot\Vert^{u}_{2}$.
Notice that for all $v=(v_{N})_{N\in\o}/\U\in H^{u}$, $\Vert v\Vert^{u}_{2}=
(\sqrt{\langle v_{N}|v_{N}\rangle_{N}})_{N\in\o}/\U$
and that $((\sqrt{\langle v_{N}|v_{N}\rangle_{N}})_{N\in\o}/\U)^{2}=\langle v\vert v\rangle^{u}$.
Similarly
for all $q=(q_{N})_{N\in\o}/\U\in\C^{u}$, we write $\vert q\vert^{u}$
for $(\vert q_{N}\vert)_{N\in\o}/\U$. Notice that
if we let $X$ be the range of this 'absolute value',
then $\R^{u}=X\cup\{-r\vert\ r\in X\}$ is a real closed field containing
the reals. In particular, it is linearly
ordered and thus we can compare e.g. the 'norm' $\Vert v\Vert^{u}$
of $v\in H^{u}$
and a rational number.
Thus for all $u,v\in H^{u}$, if there is $q\in\C$ such that
$q$ is infinitely close to $\langle u|v\rangle^{u}$, we write
$\langle u\vert v\rangle$ for this $q$. If there is no such $q$,
we write $\langle u|v\rangle=\infty$.

\begin{notation}\label{notation1.1}
When we define an element $(v_{N})_{N<\o}/\U$
of $H^{u}$, it is enough to define the vectors $v_{N}$ so that
the definition makes sense for all $N$ in some set that belongs to
$\U$. This often simplifies notation considerably.
\end{notation}

Now $H^{u}$ contains some structures of interest to us.
Let us first look at $H^{\infty}$: On $H^{u}$ we define
an equivalence relation $\sim$ so that
$u\sim v$ if $\Vert u-v\Vert^{u}_{2}$ is infinitesimal
(i.e. smaller than $1/n$ for all natural numbers $n$)
and let $H^{\infty}=H^{u}/\sim$. Now addition of elements of $H^{\infty}$ and
for all $q\in\C$, the multiplication by
$q$ are well-defined in $H^{\infty}$ and thus
$H^{\infty}$ is a vector space over $\C$.
We define a partial pairing $\langle \cdot|\cdot\rangle$ on
$H^{\infty}$ as follows: For all $u,v\in H^{\infty}$,
if there is $q\in\C$ such that for all
$u'\in u/\sim$ and $v'\in v/\sim$,
$\langle u'|v'\rangle^{u}=q$, 
then we let $\langle u|v\rangle=q$ and otherwise
$\langle u|v\rangle$ is $\infty$.

The second of these structures is the metric ultraproduct $H^{m}$:
We let $H^{m}$ be the set of all $v/\sim\in H^{\infty}$
such that $\Vert v\Vert^{u}_{2}$ is finite (i.e.
smaller than some natural number). Notice that
$H^{m}$ is still a vector space over $\C$
and $\langle \cdot|\cdot\rangle$ is total on  $H^{m}$
and, in fact, an inner product. This also gives
a norm on $H^{m}$ but it is the same as $\Vert\cdot\Vert_{2}$
defined above. Later, when considering spaces arising from other norms, 
we may write $H^{m}_{2}$ for $H^{m}$.

\begin{remark} We could define $\sim$ also on $\C^{u}$
by $q\sim r$ if $\vert q-r\vert^{u}$ is infinitesimal
and let $\C^{\infty}$  be $\C^{u}/\sim$. Notice that
multiplication is not well defined in $\C^{\infty}$
if one of the elements has infinite norm.
Then we could let $\C^{m}$   be (the field of)
all $q/\sim\in \C^{\infty}$ such that
$\vert q\vert^{u}$ is finite. Notice that
 $\C^{m}=\{ q/\sim\vert\ q\in\C\}$
and that for all $u,v\in H^{m}$, if $\langle u|v\rangle=q$, then
$q/\sim =\langle u|v\rangle^{u}/\sim$.
\end{remark}

Let $A^{u}$ be the ultraproduct of the operators $A_{N}$. Since
the operators $A_{N}$ are uniformly bounded,  $A^{u}$ induces a  well-defined operator
$A^{\infty}$ on $H^{\infty}$ and the restriction  $A^{m}$
of $A^{\infty}$ to $H^{m}$ is a well-defined operator on $H^{m}$.
We can do the same for the $A^{*}_{N}$'s and get operators $A^{*u}$, $A^{*\infty}$ and $A^{*m}$.
Notice that $(A^{m})^{*}=A^{*m}$ since
being an adjoint is preserved in metric ultraproducts and that
$A^{m}$ and $A^{*m}$ commute
since  
$A_{N}$ and $A^{*}_{N}$ commute for all $N$
and thus $A^{m}$ is normal.

Now the homormorphisms $F_N$ (from Definition \ref{def:GjaFN}) induce a homomorphism from the vector space $D(S)$ into $H^m$:
\begin{definition}
Let $F^m: D(S)\to H^m$ be given by
$$F^m(f_{P})=(F_{N}(f_{P}))_{N<\o}/\U/\sim .$$
By $F$ we will denote the homomorphism from $D(S)$ to $H^u$: 
$$F(f_{P})=(F_{N}(f_{P}))_{N<\o}/\U.$$
\end{definition}
Notice that $F^m(A_{D}(f_{P}))=A^{m}(F^m(f_{P}))$
and similarly for  $A^{*}_{D}$ and  $A^{*m}$. We can also embed $H$ into $H^m$:
\begin{definition}
Let $G^m: H\to H^m$ be given by
$$G^{m}(P(A,A^{*})(\phi ))=(P(A,A^{*})(\phi ))_{N<\o}/\U/\sim ,$$
and denote $H^{Im}=rng(G^m)$.
\end{definition}
Notice that $G^m$ is well-defined, as 
$$P(A,A^{*})(\phi )\in H_{N}$$
for all $N$
large  enough. Notice also that for all $\psi\in H$, $G^{m}(A(\psi ))=A^{m}(G^{m}(\psi ))$
and the same for $A^{*}$ and $A^{*m}$. Also
$\Vert G^{m}(\psi )\Vert^{m}_{2}=\Vert\psi\Vert_{2}$
and so $G^{m}$ is an (isometric) embedding of the Hilbert space $H$
into $H^{m}$ and maps $A$ to $A^{m}\raj H^{Im}$
and $A^{*}$ to $A^{*m}\raj H^{Im}$.

So we have seen:

\begin{lemma}\label{lemma1.2} $G^{m}$ is an isometric isomorphism from $H$
onto $H^{Im}$,  $G^{m}\circ A=A^{m}\circ G^{m}$,
$G^{m}\circ A^{*}=A^{*m}\circ G^{m}$ and
$A^{*m}$ and $A^{*m}\raj H^{Im}$ are the adjoints of  $A^{m}$
and $A^{m}\raj H^{Im}$, respectively. $\qed$
\end{lemma}

\begin{corollary}\label{corollary1.3}For all $f\in D(S)$, $G^{m}(G(f))=F^m(f)$.
\end{corollary}

\begin{proof} Now $f=f_{P}$ for some $P\in\C [X,Y]$. Then
$$G^{m}(G(f))=G^{m}(P(A,A^{*})(\phi ))=P(A^{m},A^{*m})(G^{m}(\phi ))=F^m(f),$$
where the last identity follows from \eqref{star1}, \eqref{star2} and \eqref{2stars} above.
\end{proof}

\vskip 1truecm

We now return to the reason for talking about $A$ and $A^*$ instead of $A$ and $A^{-1}$, namely there are operators that are not obtained just by multiplying a unitary operator by a constant but that still allow for the sort of approximation we developed here.

\begin{example}
Let $A=rU$ be an operator with a cyclic vector $u$ on a Hilbert space $H_1$ and $B=qV$ be an operator with a cyclic vector $v$ on a Hilbert space $H_2$, where $r,q\in\C$, $r,q\neq 0$, $r\neq q$, and $U$ and $V$ are unitary. Then consider the space $H\subset H_1\oplus H_2$ spanned by the vectors $(A\oplus B)^i(A^*\oplus B^*)^j(u\oplus v)$, for $i,j<\o$. This gives a space ($H$) and an operator $(A\oplus B)$ with a cyclic vector ($u\oplus v$) such that the operator is not a constant times a unitary operator, but one can still build spaces $H_N$ as above (generated by vectors $(A\oplus B)^i(A^*\oplus B^*)^j(u\oplus v)$ with $i,j\leq N$).
\end{example}

\begin{question}
What exactly is needed of the operator $A$ for this kind of approximation technique to work?
\end{question}

\section{Spectral measure} \label{sec:2}

In this section we show how weighted counting measures in the finite dimensional spaces give rise to a measure on subsets of $\C$. With this measure the vector space homomorphism $F$ from the previous section becomes an isometric embedding of $L_2(S,\mu)$ into the metric ultraproduct.

For all $N<\o$, we define a measure $\mu_{N}$ for subsets $X$ of $\C$:
Recall that $\phi =\sum_{n=0}^{D_{N}-1}\xi_{N}(n)u_{N}(n)$ and each $\xi_{N}(n)$ is a non-negative real. 
We let
$\mu_{N}(X)$ be the sum of all $\xi_{N}(n)^{2}$  such that $n<D_{N}$
and  $\l_{N}(n)\in X$. Note that, as $\norm{\phi}=1$, $\mu_N(X)\in[0,1]$ for all $N<\o$ and $X\subset\C$. Then we define a naive measure $\mu^{n}(X)$ (not a measure) to be
the ultralimit
$\lim_{\U}\mu_{N}(X)$ of the measures $\mu_{N}(X)$,
i.e., $\mu^{n}(X)$ is the unique real $r$ such that for all
$\e >0$, the set
$$\{ N<\o\vert\ \vert \mu_{N}(X)-r\vert <\e\}\in \U.$$

\begin{definition}\label{def:lines}
For all reals $r\in[-M,M]$ we define lines
$I_{r}=\{\l\in S\vert\ Re(\l)=r\}$ and
$J_{r}=\{\l\in S\vert\ Im(\l)=r\}$.
For all $\e >0$ we let
$I_{r}^{\e}=\{\l\in S\vert\ \vert Re(\l)-r\vert <\e\}$ and
$J_{r}^{\e}=\{\l\in S\vert\ \vert Im(\l)-r\vert <\e\}$.
For a point $\l\in S$  and $\e >0$, we let
$\l^{\e}=\{ x\in S\vert\ \vert Re(\l )-Re(x)\vert,\vert  Im(\l )-Im(x)\vert <\e\}$.
We say that a real $r\in [-M,M]$ is \emph{nice}  if for all $\d >0$ there is $\e >0$
such that $\mu^{n}(I^{\e}_{r}),\mu^{n}(J^{\e}_{r})<\d$.
\end{definition}

\begin{lemma}\label{lemma2.1}
  \begin{itemize}
\item[(i)] $\mu^{n}$ satisfies all the properties of a measure when restricted to
finite collections of sets e.g. if $X\subseteq Y$, then $\mu^{n}(X)\le \mu^{n}(Y)$
and if $X_{i}$, $i<m$, are disjoint, then $\mu^{n}(\cup_{i<m}X_{i})=\sum_{i<m}\mu^{n}(X_{i})$.

\item[(ii)] The set of $r\in [-M,M]$ that are not nice is countable.
Thus
there is a countable dense subset $NI$ of $[-M,M]$
such that every $r\in NI$  is nice and $-M,M\in NI$.
\end{itemize}
\end{lemma}

\begin{proof} (i): The claim is clear since each  $\mu_{N}$ is a measure and thus has these properties.

(ii): Suppose not.
Since $-M$ and $M$ are clearly nice,
there must be uncountably many
reals $r\in [-M,M]$ that are not nice.
Let these be $r_{i}$, $i<\o_{1}$. By symmetry
and the pigeon hole principle we may assume
that for all $i<\o_{1}$, the vertical lines witness this.
Then by the pigeon hole principle again, we may assume that there is $\d >0$
such that for all $i<\o_{1}$ and $\e >0$, $\mu^{n}(I^{\e}_{r_{i}})>\d$.
Let $K<\o$ be such that $K\d >1$. Now for all $i<K$ choose
$\e_{i}>0$ so that the sets $I^{\e_{i}}_{r_{i}}$, $i<K$,
are disjoint. But now we can find $N<\o$ such that
for all $i<K$,  $\mu_{N}(I^{\e_{i}}_{r_{i}})>\d$.
It follows that $\mu_{N}(\cup_{i<K}I^{\e_{i}}_{r_{i}})>1$, a contradiction. \end{proof}

\begin{definition}\label{def:sqstar}
Let $Sq^{*}$ be the set of all subsets $Y$ such that there are $r_{i}\in NI$
such that
$$Y=\{ \l\in S\vert\ r_{0}<Re(\l )<r_{1},\ r_{2}<Im(\l )<r_{3}\} .$$
Now define $\mu^{*}(Y)$ for any $Y\subseteq  Int(S)$ (interior of $S$)
to be the outer measure based on $\mu^{n}(X)$, $X\in Sq^{*}$,
i.e. 
$$\mu^{*}(Y)=\inf\{\sum_{i=0}^{\infty}\mu^{n}(X_{i})\vert\ X_{i}\in Sq^{*}, Y\subseteq\cup_{i<\o}X_{i}\}.$$
As usual, we say that $X\subseteq Int(S)$ is measurable if for all $Y\subseteq Int(S)$,
$\mu^{*}(Y)=\mu^{*}(Y\cap  X)+\mu^{*}(Y-X)$. For $X\subseteq S$ (or $X\subseteq\C$), we let
$\mu^{*}(X)=\mu^{*}(X\cap Int(S))$ and $X$ is measurable if $X\cap Int(S)$
is. Since there is $\e >0$ such that for all $N<\o$ and $n<D_{N}$,
$\vert \l_{N}(n)\vert <M-\e$, it is enough to look at subsets of $Int(S)$ (everything
close to the boundary of $S$ has naive measure zero). 
We write $\ol X$ for the closure of $X\subseteq \C$.
\end{definition}

\begin{lemma}\label{lemma2.2}Let $Y\in Sq^{*}$.
\begin{itemize}
\item[(i)] $\mu^{*}(Y)=\mu^{*}(\ol Y)=\mu^{n}(Y)=\mu^{n}(\ol Y)$.
\item[(ii)] $Y$ is measurable.
  \end{itemize}
\end{lemma}

\begin{proof} (i): Clearly $\mu^{*}(Y)\le\mu^{*}(\ol Y)$.
From the choice of $Sq^{*}$ it follows easily that
for all $\e >0$, there is $X\in Sq^{*}$ such that
$\vert\mu^{n}(X)-\mu^{n}(Y)\vert <\e$ and $\ol Y\subseteq X$. Thus
$\mu^{*}(\ol Y)\le\mu^{n}(Y)$ and $\mu^{n}(\ol Y)=\mu^{n}(Y)$.
Now for a contradiction, suppose that
$\mu^{*}(Y) +\e<\mu^{n}(Y)$ for some $\e >0$.
Then  we can find $Y_{i}\in Sq$, $i<\o$, such that
$Y\subseteq\cup_{i<\o}Y_{i}$ and
$\sum_{i<\o}\mu^{n}(Y_{i}) +\e<\mu^{n}(Y)$.
But then we can find $X\in Sq^{*}$ such that
$X\subseteq Y$ and $\mu^{n}(Y-\ol X)<\e$.
Since $\ol X$ is compact, there is $K<\o$ such that
$\ol X\subseteq\cup_{i<K}Y_{i}$. Now
$$\mu^{n}(Y- \ol X)+\sum_{i<K}\mu^{n}(Y_{i})<\mu^{n}(Y),$$
a contradiction by Lemma \ref{lemma2.1}.

(ii): By the definition of the outer measure $\mu^*$, and by considering coverings, it is enough to show that
if $X\in Sq^{*}$, then $\mu^{*}(X-Y)+\mu^{*}(X\cap Y)=\mu^{*}(X)$.
Clearly $\mu^{*}(X-Y)+\mu^{*}(X\cap Y)\ge\mu^{*}(X)$.
On the other hand there are $n\le 4$ and
disjoint $Z_{i}\in Sq^{*}$, $i<n$, such that
$\cup_{i<n}Z_{i}\subseteq X-Y$ and
$X-Y\subseteq\cup_{i<n}\ol Z_{i}$. Notice that $X\cap Y\in Sq^{*}$.
Now
$$\mu^{*}(X-Y)\le\mu^{*}(\cup_{i<n}\ol Z_{i})\le\sum_{i<n}\mu^{*}(\ol Z_{i})=$$
$$=\sum_{i<n}\mu^{*}(Z_{i})=\sum_{i<n}\mu^{n}(Z_{i})=(\sum_{i<n}\mu^{n}(Z_{i}))+\mu^{n}(X\cap Y)-\mu^{n}(X\cap Y)\le$$
$$\le\mu^{n}((X\cap Y)\cup\cup_{i<n}Z_{i})-\mu^{n}(X\cap Y)\le\mu^{n}(X)-\mu^{n}(X\cap Y)=\mu^{*}(X)-\mu^{*}(X\cap Y).$$
\end{proof}

Now as usual one can see that
the collection of measurable sets is closed under countable unions and complements.
In particular
every Borel set ($\subseteq S$) is measurable.
For measurable $X$ ($\subseteq S$), we write $\mu (X)=\mu^{*}(X)$ and then
$\mu$ is a  Borel measure on $S$. Notice that by Lemma \ref{lemma2.2}, $\mu (S)=\mu^{*}(S)=\mu^{n}(S)=1$
(recall equation \eqref{phi} and that $\phi$ has norm $1$).

When working with the measure and later for constructing rigged Hilbert spaces, it will be convenient to have a sequence of refinements in $Sq^*$. 

\begin{definition}
Let, for each $n<\omega$, $Sq_{n}\subseteq Sq^{*}$ be a finite set 
so that

\begin{itemize}
\item[(a)] the elements of $Sq_{n}$ are disjoint,

\item[(b)] $\cup_{X\in Sq_{n}}\ol X =S$,

\item[(c)] for all $X\in Sq^{*}$ and $m<\o$,
there are $m<n<\o$  and $Z\subseteq Sq_{n}$ such that $\ol X=\cup_{Y\in Z}\ol Y$,

\item[(d)] if $n<m<\o$, $X\in Sq_{n}$, then there is
$Z\subseteq Sq_{m}$ such that $\ol X =\cup_{Y\in Z}\ol Y$, and

\item[(e)] if $Y=\{ \l\in S\vert\ r_{0}<Re(\l )<r_{1},\ r_{2}<Im(\l )<r_{3}\}\in Sq_{n}$, then
$r_{0}<r_{1}$, $r_{1}-r_{0}<2M/(n+1)$, $r_{2}<r_{3}$ and $r_{3}-r_{2}<2M/(n+1)$.
\end{itemize}
We let $Sq=\cup_{n<\o}Sq_{n}$.
\end{definition}

\begin{remark}
Now by Lemma \ref{lemma2.2}, it is easy to see that 
for all
$X\subseteq S$,
$$\mu^{*}(X)=\inf\{\sum_{k=0}^{\infty}\mu^{n}(X_{i})\vert\ X_{i}\in Sq, X\subseteq\cup_{i<\o}\ol  X_{i}\} .$$
We will use this frequently without mentioning it.
\end{remark}

Now $L_{2}(S,\mu )$ is a  Hilbert space
and we write $\Vert  \cdot\Vert_{2}$ also for the norm of this space. 
Elements of $L_{2}(S,\mu )$ are equivalence classes
of the equivalence relation $\Vert f-g\Vert_{2}=0$
and we write $[f]_{2}$ for these equivalence classes.
Then as usual we write $C(S)$ also for $\{ [f]_{2}\vert\ f\in  C(S)\}$
and similarly for $D(S)$. 
For each equivalence class $x\in C(S)$ we fix
a continuous $f_{x}:S\rightarrow\C$ so that $x=[f_{x}]_{2}$.
We will see later (in Lemma \ref{lemma2.5} below) that this determines $f_{x}$ on $\s(A)$ and what $f_x$ does outside $\s(A)$ doesn't matter.
Now later when a function gives a value
in $C(S)$ which is an equivalence class,
we can think of it also as a continuous function from
$S$ to $\C$.
Thus $L_{2}(S,\mu )$
and $L_{2}(\s (A),\mu\raj\s(A))$ are naturally isomorphic
since every $f\in C(\s (A))$ extends to a (bounded
uniformly) continuous function $f'\in C(S)$.

Our next step is to show that the mappings $F^m:D(S)\to H^m$ and $G:D(S)\to H$ are not only vector space homomorphisms, but also isometries under the $L_2$ norm, which will let us define $F^m([f]_{2})=F^m(f)$ and $G([f]_{2})=G(f)$ to get isometric embeddings of $L_2(S,\mu)$ into $H^m$ and $H$ respectively.

\begin{lemma}\label{lemma2.3} For all  $P\in\C [X,Y]$,
$\Vert F(f_{P})\Vert_{2}=\Vert f_{P}\Vert_{2}$.
\end{lemma}

\begin{proof} This proof is essentially the same as the proof
that shows that continuous functions are Riemann integrable.
Let $1>\e >0$, $m=1+ \sup\{\vert f_{P}(\l )\vert\ \vert\ \l\in S\}$
and choose $n>0$ so that  if $\vert\l -\l'\vert <4M/n$ (which happens for all $\lambda,\lambda'\in\ol Y$, for $Y\in Sq_n$), then
$\vert f_{P}(\l )-f_{P}(\l')\vert <\e /8m$. Then we can choose
$Z\in  \U$ such that for all $N\in Z$ and $Y\in Sq_{n}$,
both $\vert\mu (Y)-\mu_{N}(Y)\vert$ and $\vert\mu (\ol Y)-\mu_{N}(\ol Y)\vert$ are $ <\e /(4m^{2}\vert Sq_{n}\vert )$. %
Notice that if we let $X$ be the union of all
$\ol Y-Y$, $Y\in Sq_{n}$, then $\mu(X)=0$, and
$$\norm{\sum_{\l_{N}(k)\in X}\xi_{N}(k)f_{P}(\l_{N}(k))u_{N}(k)}_2^2\leq\mu_N(X)m^2<\e/4.$$

Then
\begin{multline*}
\vert\ (\Vert F_{N}(f_{P})\Vert_{2})^{2}-(\Vert f_{P}\Vert_{2})^{2}\  \vert 
 =\vert\ \sum_{k<D_N}\xi_N(k)^2\abs{f_P(\l_N(k))}^2 - \int_S\abs{f_P(x)}^2d\mu \ \vert\\
 =\vert\ \sum_{\l_N(k)\in X}\xi_N(k)^2\abs{f_P(\l_N(k))}^2 +\sum_{Y\in Sq_n}\sum_{\l_N(k)\in Y}\xi_N(k)^2\abs{f_P(\l_N(k))}^2\\
 - \sum_{Y\in Sq_n}\int_Y\abs{f_P(x)}^2d\mu \ \vert\\
 \leq \e/4 +\sum_{Y\in Sq_n}\abs{\sum_{\l_N(k)\in Y}\xi_N(k)^2\abs{f_P(\l_N(k))}^2 - \int_Y\abs{f_P(x)}^2d\mu},
\end{multline*}
and for each $Y\in Sq_n$
\begin{multline*}
    \abs{\sum_{\l_N(k)\in Y}\xi_N(k)^2\abs{f_P(\l_N(k))}^2 - \int_Y\abs{f_P(x)}^2d\mu}\\
    \leq \int_Y \left(\frac{\sum_{l_N(k)\in Y}\xi_N(k)^2\abs{\abs{f_P(\l_N(k))}^2-\abs{f_P(x)}^2}}{\mu(Y)}\right.\\
\left.\quad    +\abs{\frac{\sum_{\l_N(k)\in Y}\xi_N(k)^2\abs{f_P(x)}^2}{\mu(Y)}-\abs{f_P(x)}^2}\right)d\mu\\
    \leq \int_Y \left(\frac{\mu_N(Y)2m\frac{\e}{8m}}{\mu(Y)}+\abs{\frac{\mu_N(Y)}{\mu(Y)}-1}m^2\right)d\mu\\
    =\mu_N(Y)\frac{\e}{4}+\abs{\mu_N(Y)-\mu(Y)}m^2\\
    < \mu_N(Y)\frac{\e}{4}+\frac{\e}{4\abs{Sq_n}}.
\end{multline*}
Putting this together (and recalling $\mu_N(S)=1$), we get
$$
\vert\ (\Vert F_{N}(f_{P})\Vert_{2})^{2}-(\Vert f_{P}\Vert_{2})^{2}\  \vert 
\le\e/4 +\e /4 +\e/4<\e 
$$
from which the claim follows.
\end{proof}

\begin{corollary}\label{cor:Fm_and_G_isometric}
$F^m([f]_{2})=F^m(f)$ and $G([f]_{2})=G(f)$ are isometric embeddings of $L_2(S,\mu)$ into $H^m$ and $H$ respectively.
\end{corollary}
\begin{proof}
  By the above, and noting that $G=(G^{m})^{-1}F^m$, we see that $F^m$ and $G$ are well-defined on the $\norm{\cdot}_2$ equivalence classes of $D(S)$. 
  Now we notice  that $\C [X,Y]$ is closed under the function
$f\mapsto f^{*}$, where $f^{*}(\l )=\overline{f(\l )}$. And thus
by Stone-Weierstrass,
$D(S)$ is dense in  $C(S)$ even in the uniform convergence
topology (we will need this later).  And since $C(S)$
is dense in $L_{2}(S,\mu )$, see, e.g.,  \cite{Bo} Corollary 4.2.2,
$D(S)$ is dense in $L_{2}(S,\mu )$.
Thus we can extend the functions $F^m$  and
$G$ to all of $L_{2}(S,\mu )$ by continuity.
\end{proof}

\begin{remark}
Recall that on $D(S)$, we have $A_D(f_P)(\l)=\l f_P(\l)$ and $A^*_D(f_P)(\l)=\ol \l f_P(\l)$. These can also be defined on the $\norm{\cdot}_2$ equivalence classes by
 $A_{D}([f]_{2})=[A_{D}(f)]_{2}$ (and similarly for $A^{*}_{D}$), and then extend it to all of $L_2(S,\mu)$ by continuity.
\end{remark}

\begin{corollary}\label{corollary2.4} 
$F^m$ is  an isometric isomorphism from
$L_{2}(S,\mu )$ to $H^{Im}$ and $A^{m}\raj H^{Im}=F^mA_{D}(F^m)^{-1}$.
Thus also $G=(G^{m})^{-1}F^m$ is an isometric isomorphism from
$L_{2}(S,\mu )$ to $H$ and $A=GA_{D}G^{-1}$. $\qed$
\end{corollary}

We have been looking at $L_2(S,\mu)$, where $S$ was chosen to be a 'big enough' set. However, $\mu$ only puts weight on part of this set, and next we will show that that set is exactly the spectrum of $A$.
As pointed out above, there is a rational number  $0<q<M$
such that for all $N<\o$, the norm of $A_{N}$ is $<q$  since
the norm of $A_{N}$ is  at most the norm of $A$. Thus
if we let $S^{-}=\{\l\in\C\vert\ -M< Re(\l )<M, -M< Im(\l )<M\}$
then $\mu (S^{-})=\mu (S)=1$.

\begin{definition}
Let $\s (A)$ be the set of all $\l\in\C$ such that
$A-\l I$ does not have a bounded inverse.
Notice that $\l\in\s (A)$ iff for all $\e >0$, there is
$\psi\in H$ such that
$\Vert\psi\Vert_{2}=1$ and
$\Vert A(\psi )-\l\psi\Vert_{2}<\e$.
We let $\s^{*}(A)$
be the set of all $\l\in S^{-}$ such that every open neighbourhood of
$\l$ that is contained in $S^{-}$ has measure $>0$.
\end{definition}

\begin{lemma}\label{lemma2.5} $\s (A)=\s^{*}(A)$.
\end{lemma}

\begin{proof} By Corollary \ref{corollary2.4}, it is enough to show that
$\s (A_{D})=\s^{*}(A)$, where
$\s (A_{D})$ is defined as $\s (A)$ was.
Also if $\l\not\in S^{-}$, $\l$ does not belong to either of the sets.
So suppose that $\l\in S^{-}$.

We write $B_{r}(\l )$, $r$ a positive real number, for the set of
all $\l'\in\C$ such that $\vert\l -\l'\vert<r$.
Now if for some $r$,  $B_{r}(\l )\subseteq S^{-}$  and $\mu (B_{r}(\l ))=0$, then
it is easy to see that $\Vert A_{D}(f)-\l f\Vert_{2}\ge r$
for all $f$ with norm $1$
and thus $\l\not\in\s (A_{D})$. On the other hand, if for all positive $r$
such that $B_{r}(\l )\subseteq S^{-}$,  $\mu  (B_{r}(\l  ))>0$, then
for  all such $r$, we  can find  $f\in C(S )$
such that $\Vert  f\Vert_{2}=1$ and $f(x)=0$ for all
$x\in S-B_{r}(\l )$. Then $\Vert A_{D}(f)-\l f\Vert_{2}<r$.
Thus $\l\in\s(A_{D})$. \end{proof}

\begin{corollary}\label{corollary2.6} $\mu (\s (A))=1$ and $\mu (S-\s (A))=0$.
\end{corollary}

\begin{proof} $\s (A)=\s^{*}(A)$ is a closed set and thus measurable and thus
the claims follow from  Lemma \ref{lemma2.5} and $\mu(S)=\mu(S^-)=1$. \end{proof}

Looking back at the constructions of this and the past section, one notes that the finite dimensional spaces $H_N$ depend on the choice of $\phi$, and thus it is likely that one gets different $H^u$ and $H^m$ depending on the choice of cyclic vector. We have not looked into this, as the whole of these spaces are too big for our interest. As every $H_N$ has a part, where the operator is replaced by a wraparound-mapping, these erroneous parts will show up in the ultraproduct, and we are really interested only in parts stemming from the original operator. There are subspaces where these errors disappear, and one of them is  $H^{Im}$, which is canonical, as we have showed it to be isometrically isomorphic to $H$. A natural question arises:

\begin{question}
Is there a natural characterization for the subspaces $H^m$ where the errors of the $A_N$'s disappear?
\end{question}

As for the canonicity of our measure $\mu$ we have the following:

\begin{corollary}
The measure $\mu$ constructed is independent of the choice of cyclic vector. In particular all measures constructed in the above manner will agree on measure zero sets, and thus agree on which lines are nice.
\end{corollary}

\begin{proof}
  This follows from Corollary \ref{corollary2.4} and the uniqueness of the spectral measure (see, e.g., section VII.2 of \cite{RS}.
\end{proof}

However, note that in general $\sigma(A)\neq\sigma(A^m)$, as the latter can be strictly larger. Thus we are not considering elementary embeddings into the (metric) ultraproduct, but only embeddings. This is enough for our purposes, as we only use the construction for calculating equations.

\section{Self-adjoint operators}\label{sec:3}

In this section we show that the constructions from Sections \ref{sec:1} and \ref{sec:2} can be generalized to self adjoint operators.

We will need the following facts, some of which seem hard to
find in the literature (in particular item (iv) of the fact).
Our version is not the best possible.

\begin{fact}\label{fact3.1} Let $H$ be a  separable Hilbert space and
$B$ a self-adjoint, $U$ a unitary operator on $H$ and $I$ the identity operator.
\begin{itemize}
\item[(i)] $\max\{\vert\l\vert \,\vert\, \l\in\s(B)\}=\Vert B\Vert$,  $e^{iB}$ is unitary and $\s(e^{iB})=\{e^{i\l} \vert \l\in\s(B)\}$.
\item[(ii)] Let $r$ and $q$ be reals.
$e^{i(B+r I)}=e^{iB}\circ e^{ir}I$,
$e^{iq}\in\s (e^{iB})$ iff $e^{i(q +r)}\in\s (e^{i(B+r I)})$
and $\l\in\s (B)$ iff $\l +r\in\s (B+rI)$.
\item[(iii)] Suppose $H_{0}$ is a complete subspace of $H$,
$H_{1}$ its orthogonal complement and $A_{i}$ are bounded operators
on $H_{i}$, $i<2$. Then there is a unique operator $A$ on
$H$ such that $A\raj H_{i}=A_{i}$ for both $i<2$
and it is bounded and $\s (A)=\s (A_{0})\cup\s (A_{1})$.
If in addition
both operators $A_{i}$ are self-adjoint,  then also
$A$ is self-adjoint and
if both are unitary, then also $A$ is unitary.
\item[(iv)] If $0\le r<\pi /2$ and for all $\l\in\s (U)$,
$\l=e^{iq}$ for some $q\in [-r,r]$, then there is
a unique self-adjoint operator $A$ such that
$U=e^{i A}$ and $\s (A)\subseteq [-r,r]$.
 \end{itemize}
\end{fact}

\begin{proof} (i): The first claim is the spectral radius formula (for a proof see, e.g., Chapter VI of \cite{RS}) the others are proved via polynomial approximations, see Chapter VII of \cite{RS}. %

(ii): The first claim follows from the fact that
$B$ and $r I$ commute. The rest is easy.

(iii): Obviously $A$ must be such that
if $v\in H_{0}$ and $u\in H_{1}$, then
$A(v+u)=A_{0}(v)+A_{1}(u)$ and this is a bounded operator.
Also clearly $\s (A_{0})\cup\s (A_{1})\subseteq\s (A)$
and since  spectra are closed the other direction follows by
an easy calculation.  If $A_{0}$ and $A_{1}$ are self adjoint,
then for all $v_{0},v_{1}\in H_{0}$ and $u_{0},u_{1}\in H_{1}$,
$$\langle A(v_{0}+v_{1})|u_{0}+u_{1}\rangle=\langle A_{0}(v_{0})|u_{0}\rangle+\langle A_{1}(v_{1})|u_{1}\rangle=$$
$$=\langle v_{0}|A_{0}(u_{0})\rangle+\langle v_{1}|A_{1}(u_{1})\rangle=\langle v_{0}+v_{1}|A(u_{0}+u_{1})\rangle$$
where the first and the third identity follow from the choice
of $H_{0}$ and $H_{1}$. Finally if $A_{0}$
and $A_{1}$ are unitary, an easy calculation shows that
$A$ is norm preserving.

(iv): Let $U'=e^{i\pi}U$. Then
$U'$ is unitary, and if $\l\in\s (U')$, then
$\l =e^{iq}$ for some $q\in [-r+\pi,r+\pi]$. Now by
\cite[Theorem 8.4]{St}, there is a unique self-adjoint operator
$A'$ such that $U'=e^{i2\pi A'}$ and the spectrum
of $A'$ is a subset of $[0,1]$ and 0 is not in the point spectrum of $A'$. By item (i), the spectrum of $2\pi A'$ is a subset of $[\pi-r,\pi+r]$. Let $A=2\pi A'-\pi I$.
Then $e^{iA}=U$,
$\s (A)\subseteq [-r,r]$ and the uniqueness
follows from the uniqueness of $A'$ and (ii).  \end{proof}

\vskip 1truecm

From now on, in this section we assume that $Re(e^{i})\in NI$.
If this is not true, we can replace $1$ by a real $q$ close to $1$
such that $Re(e^{iq})\in NI$
and work with it.

Let $B$  be a bounded self-adjoint operator on a separable  Hilbert space $H$.
It will turn out that it does not matter whether we look at $B$ or $rB$
for some positive real $r$, and thus we may assume that the norm of
$B$ is $<1$. Also we may assume that $B(v)\ne 0$ for all
$v\in H-\{ 0\}$ because we can always restrict $B$ to the orthogonal
complement of its kernel (which is closed under $B$ since $B$ is
self-adjoint).
Let $A=e^{iB}$. Then $A$ is unitary
and not the identity by Fact \ref{fact3.1}.
Also by the same fact, if $\l\in\s (A)$, then $\l =e^{iq}$ for some
$q\in [-1,1]$.

In the proof of the following lemma, the assumption that the norm of $B$ is small plays an important role, stemming from Fact \ref{fact3.1}. We will need the uniqueness of the self-adjoint operator in point (iv) of the fact, and for that we need to avoid a 'wraparound'. Our formulation of the fact allows for $B$ to have norm smaller that $\pi/2$, and we have chosen 1 as it is convenient in calculations.

\begin{lemma}\label{lemma3.2} If $H_{0}$ is a complete subspace of $H$ and closed under $A$ and $A^{*}$,
then it is closed under $B$.
\end{lemma}

\begin{proof} Suppose not. Let
$H_{1}$ be the orthogonal
complement of $H_{0}$ in $H$. Then $H_{1}$
is closed under  $A$ and $A^{*}$
and there are
self-adjoint operators $B_{0}$ in $H_{0}$
and $B_{1}$ in $H_{1}$ such that
$e^{i B_{0}}=A\raj H_{0}$ and  $e^{i B_{1}}=A\raj H_{1}$.
But then  by Fact \ref{fact3.1} (iii) there is a  self-adjoint operator
$B'$ in $H$ such that $e^{i B'}=A$
(because $e^{i B'}\raj H_{0}=A\raj H_{0}$
and $e^{i B'}\raj H_{1}=A\raj H_{1}$
and $A$ is uniquely determined by $A\raj H_0$ and $A\raj H_1$)
and $B'\ne B$.
This contradicts Fact \ref{fact3.1} (iii).\end{proof}

By Lemma \ref{lemma3.2}, if we may assume that $H$ has a cyclic vector
with respect to the operator $A$ as in Section \ref{sec:1},
 we have spaces $H_{N}$, operators  $A_{N}$, functions
$G$, $G^{m}$, $F^m$ and $F_{N}$, a spectral measure $\mu$ etc.
as in the first two sections.
Notice that each $A_{N}$ is a unitary operator and thus for all $N<\o$
and $n< D_{N}$, $\l_{N}(n)=e^{iq}$ for some
$q\in ]-\pi ,\pi ]$.
It follows that $\mu (S-S_{0})=0$
when $S_{0}$ is the set  $\{ e^{iq}\vert\ q\in ]-\pi ,\pi]\}$.

We define operators $B_{N}$ as  follows:
$B_{N}(u_{N}(n))=qu_{N}(n)$ if
$A_{N}(u_{N}(n))=e^{iq}u_{N}(n)$ and $q\in ]-\pi ,\pi]$.
We write $B^{m}$ for the metric ultraproduct
(with the filter $\U$ as in the previous sections)
of the operators
$B_{N}$. Then $A_{N}=e^{iB_{N}}$ for all $N<\o$ and
since the operators $B_{N}$ are uniformly bounded (by $\pi$)
it is easy to see that $A^{m}=e^{iB^{m}}$. Also each $B_{N}$ is self-adjoint
and thus also $B^{m}$ is self-adjoint.
We want to show that

\begin{thm}\label{theorem3.3} $B^{m}\raj H^{Im}=G^{m}B(G^{m})^{-1}$.
\end{thm}

This follows immediately from  Fact \ref{fact3.1} (iv)
if we know that $H^{Im}$ is closed under $B^{m}$.
So what we need to show is the following:

\begin{lemma}\label{lemma3.4} $H^{Im}$ is closed under $B^{m}$.
\end{lemma}

\begin{proof} Suppose not. Note that from Corollaries \ref{corollary1.3} and \ref{cor:Fm_and_G_isometric} it follows that $H^{Im}:=rng(G^m)=rng(F^m)$. Let $H_{0}$ be the closure of $H^{Im}$ under $B^{m}$
i.e. the complete subspace of $H^{m}$ generated by the set
$$\{ (B^{m})^{n}(F^m(f_{P}))\vert\ P\in \C[X,Y],\ n<\o\}$$
(keep in mind that $B^{m}$ is bounded).
Then $H^m$ is a complete subspace of $H_0$, closed under $A^m$ and $A^{m*}$, so we get a contradiction as in the proof of Lemma \ref{lemma3.2}, unless the norm is too large to guarantee uniqueness. But this we can rule out by showing the following:

\begin{claim}\label{claim3.4.1} The operator norm of $B^{m}\raj H_{0}$ is $\le 1$.
  \end{claim}

\begin{proof} For a contradiction suppose that there is
$\e >0$, $k<\o$ and $P\in\C [X,Y]$ such that
$\Vert B^{m}((B^{m})^{k}(F^m(f_{P})))\Vert >(1+\e )\Vert (B^{m})^{k}(F^m(f_{P}))\Vert$.
Since $A$ is not the identity, $\mu(\{ 1\} )<1$.
So there are $\d >0$  and $V\subseteq\s (A)$
such that $\mu (V)>\d$, for all $\l\in V$,
$\vert f_{P}(\l )\vert >\d$ and if $\l =e^{iq}\in V$,
then $\vert q\vert >\d$.

For all $N<\o$, let $H^{0}_{N}$
be the subspace of $H_{N}$ generated by those $u_{N}(n)$
for which $\l_{N}(n)=e^{iq}$ for some  $q\in [-\pi ,\p]$
such that $\vert q\vert >1$ and let $H^{1}_{N}$
be the subspace of $H_{N}$ generated by those $u_{N}(n)$
for which $\l_{N}(n)=e^{iq}$ for some  $q\in [-\pi ,\p]$
such that $\vert q\vert\le 1$. Then
$H^{1}_{N}$ is the orthogonal complement of $H^{0}_{N}$. Further let $H^2_N$ be the subspace of $H_N$ generated by those $u_N(n)$ for which $\lambda_N(n)\in V$.
Let $v_{i}(N)$ be the projection of $F_{N}(f_{P})$ to $H^{i}_{N}$, $i<3$.
Let $S^{*}$ be the set of all $\l\in S$ such that
$Re(\l )\ge Re(e^{i})$.
Notice that since $Re(e^{i})\in NI$
(see the assumption immediately after the proof of Fact \ref{fact3.1}),
$\mu (S^{*})$ and $\mu (S-S^{*})$
are ultralimits of
$\mu_{N}(S^{*})$ and $\mu_{N}(S-S^{*})$.
Let  $Z$  be the supremum of all $\vert f_{P}(\l )\vert$ for $\l\in S$.

Now
\begin{itemize}
    \item[(a)] $\norm{(B_N)^k(F_N(f_P))}\geq\norm{(B_N)^k(v_2(N))}=$
    $$=\norm{(B_N)^k\sum_{\l_N(n)\in V}\xi_N(n)f_P(\l_N(n))u_N(n)}\geq \d^{k+3/2}$$
in a set
of $N$'s that belongs to  $\U$,
\item[(b)] $\norm{B_N((B_N)^k(F_N(f_P)))}=\norm{(B_N)^{k+1}(v_0(N)+v_1(N))}\leq$
$$\pi^{k+1}\sqrt{\mu_N(S-S^*)}Z+\norm{v_1(N)}$$
in a set of $N$'s that belongs to $\U$,
\item[(c)] $\mu_{N}(S-S^{*})\le (\e\d^{k+3/2}/(2\pi^{k+1}Z))^{2}$
in a set of $N$'s that belong to $\U$ (since $\mu (S-S^{*})=0$ by the assumption $\norm{B}\leq 1$ and the construction of the measure for $A$).
\end{itemize}
By putting these together, we get
$$
\norm{B_N((B_N)^k(F_N(f_P)))}\leq \pi^{k+1}\sqrt{\mu_N(S-S^*)}Z+\norm{v_1(N)}\leq (frac{\e}{2}+1)\norm{(B_N)^k(F_N(f_P))}
$$
in a set of $N$'s that belongs to $\U$, which is a contradiction.
\renewcommand{\qedsymbol}{}
\end{proof}\qed Claim \ref{claim3.4.1}

\end{proof}

Now we can define an operator $B_{D}$ on $L_{2}(S,\mu )$ the natural way:
For $f\in C(S)$, we let
$B_{D}(f)(\l )=K_B(\l )f(\l )$, where
$K_B(\l )=q$ if $q\in ]-\pi ,\pi ]$ and
$\l =pe^{iq}$ for some real number $p$.
Now it is easy to see that $e^{iB_{D}}=A_{D}$
and thus by Fact \ref{fact3.1},
$B^{m}\raj H^{Im}=F^m B_{D}(F^m)^{-1}$.

If we want for $B$ exactly the same setup as for $A$ we can do the following:
Let $S_{B}=]-\pi ,\p]$ and define a Borel measure $\mu_{B}$ on $S_{B}$ by letting
$\mu_{B}(X)=\mu (K_{B}^{-1}(X))$. Then we can define an operator
$B'_{D}$ on $L_{2}(S_{B},\mu_{B})$ so that for all
$f\in C(S_{B})$, $B'_{D}(f)(q)=qf(q)$.
Finally we notice that if we define a continuous
$F_{B}: L_{2}(S,\mu )\rightarrow L_{2}(S_{B},\mu_{B})$
so that for $f\in C(S)$,
$F_{B}(f)(q)=f(e^{iq})$ we get an isometric isomorphism
and $B'_{D}=F_{B}B_{D}F_{B}^{-1}$.

\section{Rigged Hilbert spaces}\label{sec:4}

Rigged Hilbert spaces are a way of accounting for objects that don't fit into a Hilbert space, such as Dirac delta functions. A rigged Hilbert space consists of a Hilbert space $H$ and a subspace $\Phi$ (with a finer norm topology than the one given by the norm in $H$). One then has
$$
\Phi\subset H\subset \Phi^*,
$$
where $\Phi^*$ is the dual space of $\Phi$. The elements in $\Phi^*$ are antilinear functionals, and if $\Phi$ was suitably chosen with respect to a cyclic, essentially self adjoint operator $B$, then the functionals act as generalized eigenvectors corresponding to values in the spectrum of $B$ (see, e.g., \cite{Ma}).

In this section we set out to find something similar to $\Phi^*$ by considering different norms in our ultraproduct space $H^u$, and finding extensions of $H^{Im}$.
Here, as  well as in the next section, it does not matter whether we look at
a self-adjoint bounded
$B$ or $e^{iB}$, and thus we return to the  assumptions of the first two sections,
i.e., that $A$ is bounded, $A\circ A^{*}=A^{*}\circ A=r_{A}I$ and $\phi$
is a cyclic vector. Also we let $H_{N}$, $A_{N}$ etc. be as in these sections.

We will look at two norms, that end up being essentially  an $L_{\infty}$-norm and
an $L_{1}$-norm. The latter of these will be called
$\Vert\cdot\Vert_{0}$ in order to avoid confusion. Strictly speaking
both of these may be only seminorms in the
finite dimensional spaces $H_{N}$, but we will call them norms for simplicity. In the ultraproduct we will  fraction out
all vectors to which the seminorm gives an infinitesimal value, and then
the seminorms become norms, and we obtain spaces $H^m_0$ and $H^m_\infty$ somewhat similarly to how we obtained $H^m$ for the $\norm{\cdot}_2$-norm before. Then we will show how generalized distributions embed into $H^m_0$ such that values of the distribution can be calculated via a pairing function resembling an inner product.

\begin{definition}
Consider the following functions
from $H_{N}$ to $\C$:
Let $X\subseteq\C$ be a closed set such that $X\subseteq S$. Then
$$\Vert\sum_{n=0}^{D_{N}-1}a_{n}u_N(n)\Vert^{X}_{\infty}=\sup\{\xi_{N}(n)^{-1}\vert a_{n}\vert\ \vert\ \l_{N}(n)\in X\} ,$$
where by $0^{-1}$ we mean $0$
and
$$\Vert\sum_{n=0}^{D_{N}-1}a_{n}u_N(n)\Vert_{0}=\sum_{n=0}^{D_{N}-1}\xi_{N}(n)\vert a_{n}\vert .$$
Keep in  mind that for all $N$ and $n$, $\xi_{N}(n)$ is a non-negative real number $\le 1$.
\end{definition}

When constructing $H^m$ from the $\norm{\cdot}_2$-norm we just fractioned out infinitesimals and discarded elements of infinite length. For our other norms we need to be a bit more careful.
We start with $\Vert \cdot\Vert_{0}$.
As with the Hilbert space norm, this gives
first a function from $H^{u}$ to $\C^{u}$,
denoted by $\Vert \cdot\Vert^{u}_{0}$
(as before),
and an equivalence  relation $\sim_{0}$.

\begin{definition}\label{def:good}
Let $n<\o$, $Y\subseteq Sq_{n}$, $R^{Y}_{0}=\cup_{R\in Y}R$, $S^{Y}_{0}=\cup_{R\in Y}\ol R$,
$R^{Y}_{1}=\cup_{R\in Sq_{n}-Y}R$ and $S^{Y}_{1}=\cup_{R\in Sq_{n}-Y}\ol R$.
We say that $Y$ is \emph{good}
if the distance between $\s (A)$ and the closure of $S^{Y}_{1}$
is $>0$ (then $\s (A)\subseteq S^{Y}_{0}$). We say that
$$(\sum_{k=0}^{D_{N}-1}a^{N}_{k}u_{N}(k))_{N<\o}/\U\in H^{u}$$
(and $(\sum_{k=0}^{D_{N}-1}a^{N}_{k}u_{N}(k))_{N<\o}$)
is \emph{$0$-good} if for all good $Y$
$$(\sum_{\l_{N}(k)\in S^{Y}_{1}}\xi_{N}(k)\vert a^{N}_{k}\vert )_{N<\o}/\U/\sim =0.$$

Then we let  $H^{m0}$  be the set of all $0$-good
$f/\U\in H^{u}$ such that
$\Vert f/\U\Vert^{u}_{0}<n$ for some natural number $n$.
And we let $H^{m}_{0}=H^{m0}/\sim_{0}$.
Finally, we define the norm $\Vert \cdot\Vert_{0}$
on $H^{m}_{0}$ as before: 
$\Vert f/\sim_{0}\Vert_{0}=q$ if $q\in\C$ is the unique
element such that for all $\e >0$,
$$\{  N<\o\vert\ \vert\Vert f(N)\Vert_{0}-q\vert <\e\}\in \U.$$
\end{definition}

With this norm $H^{m}_{0}$ is a Banach space
(addition and scalar  multiplication are defined as before).
In order to simplify the notation, for all
$f\in\Pi_{N\in\o}H_{N}$,  we write
$[f]_{0}=f/\U/\sim_{0}$.
Notice that if
$f,g\in H^{u}$ and $\Vert f-g\Vert^{u}_{0}$ is infinitesimal
and $f$ is $0$-good,
then so is $g$. Notice that for all $f\in C(S)$, $F(f)\in H^{m0}$
and $\Vert f\Vert_{1}=\Vert [F(f)]_{0}\Vert_{0}$: The first of the claims is immediate
since $f$ is bounded and $\mu(S)$ finite, and  the second follows from the fact that $f$
is uniformly continuous and thus Riemann integrable.

Let us then look at the norm $\Vert \cdot\Vert_{\infty}$
First we will need a modified version of the $\sup$-norm on $C(S)$:

\begin{definition}
  We define a seminorm
$\Vert f\Vert^\s_{\infty}=\sup\{\vert f(\l )\vert\ \vert\ \l\in\s (A)\}$
on $C(S)$ and
write $L_{\infty}(S,\mu )$ for $C(S)/\sim^\s_{\infty}$,
where $f\sim^\s_{\infty}g$ if $\Vert f-g\Vert^\s_{\infty}=0$.
\end{definition}

Then $\Vert \cdot\Vert^\s_{\infty}$ is a norm on $L_{\infty}(S,\mu )$.

\begin{definition}
For all $\e >0$, we let $X_{\e}$ be the set of all $\l\in\C$
whose distance to $\s (A)$ is $\le\e$ and we let
$\Vert (v_{N})_{N<\o}/\U\Vert^{\e}_{\infty}$ be the unique $q\in\C$
such that for all $\d >0$
$$\{ N<\o\vert\ \vert\Vert v_{N}\Vert^{X_{\e}}_{\infty}-q\vert <\d\}\in \U$$
if there is such a $q$ and otherwise the value is $\infty$
(i.e. we do as before).

Now let $v=(\sum_{k=0}^{D_{N}-1}a^{N}_{k}u_{N}(k))_{N<\o}/\U\in H^{u}$.
Then we let
$\Vert v\Vert_{\infty}=\lim_{\e\rightarrow 0}\Vert v\Vert^{\e}_{\infty}$
if there are $p<\o$ and $Z\in \U$ for which
$\vert a^{N}_{k}(\xi_{N}(k))^{-1}\vert <p$ for all $N\in Z$ and $k<D_{N}$
and otherwise we let $\Vert v\Vert_{\infty}=\infty$. Notice that thus having bounded $\Vert\cdot\Vert_\infty$-norm in $H^u$ requires the $\Vert\cdot\Vert^S_\infty$-norm to be bounded in a $\U$-large set of $H_N$s although the value is determined by what happens 'close to $\sigma(A)$'.
We then let $H^{m\infty}$ be the set of all
$v\in H^{u}$ for which $\Vert v\Vert_{\infty}<\infty$ and
$H^{m}_{\infty}=H^{m\infty}/\sim_{\infty}$, where $\sim_{\infty}$ is the relation
$\Vert v-u\Vert_{\infty}=0$.
We write $[f]_{\infty}$ for $f/\U/\sim_{\infty}$.
\end{definition}

\begin{lemma}\label{lemremark4.1} \begin{itemize}
    \item[(i)] For all $f\in C(S)$, $F(f)\in H^{m\infty}$ and
      $\Vert f\Vert^\s_{\infty}=\Vert F(f)\Vert_{\infty}$.
      \item[(ii)] Every element of $H^{m\infty}$ has finite $\norm{\cdot}_2$-norm, and for all $v\in H^u$, $\norm{v}^u_0\leq\norm{v}_2$. For $f\in C(S)$, $\norm{F(f)}_0\leq\norm{F(f)}_2\leq \norm{F(f)}_\infty$.
\end{itemize}
      \end{lemma}
\begin{proof}
(i)  The first claim is clear since $f$
is bounded. For the second, we notice 
that $f$ is uniformly continuous
and thus
the direction $\Vert f\Vert^\s_{\infty}\ge\Vert F(f)\Vert_{\infty}$ follows easily.
The direction
$\Vert f\Vert^\s_{\infty}\le\Vert F(f)\Vert_{\infty}$ follows from the equality of $\s(A)$ and $\s^*(A)$
and the continuity of $f$.

(ii) The first claim is a straight forward calculation, the second follows from Hölder's inequality, as $\mu(S)=1$. The chain of inequalities follows from the classical result.
\end{proof}

Let $H^{m}_{*}=\bigcup_{k\in\{ 0,2,\infty\}}H^{m}_{k}$. We define
a pairing $\langle \cdot|\cdot\rangle$  on $H^{m}_{*}$ the same way as before:
For $[f]_{k},[g]_{r}\in H^{m}_{*}$, $\langle [f]_{k}|[g]_{r}\rangle=q$ if
$q\in\C$ is such that for all $f'\in [f]_{k}$  and
$g'\in[g]_{r}$, $\langle f'|g'\rangle^{u}$ is infinitely close to $q$. 
If no such $q$ exists we let $\langle [f]_{k}|[g]_{r}\rangle=\infty$
(meaning that $\langle [f]_{k}|[g]_{r}\rangle$ is not well-defined).

\begin{lemma}\label{lemma4.2} If $x\in H^{m}_{\infty}$ and $y\in H^{m}_{0}$,
then $\langle x|y\rangle\ne\infty$.
\end{lemma}

\begin{proof} Let $x=[(x_{N})_{N<\o}]_{\infty}$,
where $$x_{N}=\sum_{n=0}^{D_{N}-1}a^{N}_{n}u_{N}(n)$$
and $y=[(y_{N})_{N<\o}]_{0}$,
where $$y_{N}=\sum_{n=0}^{D_{N}-1}b^{N}_{n}u_{N}(n)$$
and let $q>\Vert x\Vert_{\infty}$ and $r>\Vert y\Vert_{0}$.
Let $p<\o$,
be such that for all $N<\o$ and $n<D_{N}$,
$\vert a^{N}_{n}(\xi_{N}(n))^{-1}\vert <p$.
It is enough to show that
$$\vert \langle (x_{N})_{N<\o}|(y_{N})_{N<\o}\rangle\vert <qr.$$

Let $\e>0$ be such that $\Vert x\Vert^{\e}_{\infty}<q$.
Then we choose $n<\o$ so that there is a good $Y\subseteq Sq_{n}$ such that
if as before we denote $S^{Y}_{0}=\cup_{R\in Y}\ol R$ and
$S^{Y}_{1}=\cup_{R\in Sq_{n}-Y}\ol R$, then $S-X_{\e}\subseteq S^{Y}_{1}$.
Then
$$\vert \langle x_{N}|y_{N}\rangle\vert\le\sum_{n=0}^{D_{N-1}}\vert a_{n}^{N}\vert\ \vert b^{N}_{n}\vert <$$
$$< \sum_{\l_{N}(n)\in S^{Y}_{0}}\xi_{N}(n)q\vert b^{N}_{n}\vert +\sum_{\l_{N}(n)\in S^{Y}_{1}}\xi_{N}(n)p\vert b^{N}_{n}\vert\le$$
$$\le q\sum_{\l_{N}(n)\in S^{Y}_{0}}\xi_{N}(n)\vert b^{N}_{n}\vert +p\sum_{\l_{N}(n)\in S^{Y}_{1}}\xi_{N}(n)\vert b^{N}_{n}\vert
<qr +p\sum_{\l_{N}(n)\in S^{Y}_{1}}\xi_{N}(n)\vert b^{N}_{n}\vert $$
in a set of $N<\o$ that
belongs to $\U$. Now the claim follows from the assumption that $(y_{N})_{N<\o}/\U$ is $0$-good.
\end{proof}

Next we set out to find distributions as vectors in our space $H^u$. 

\begin{definition}
We say that an antilinear map $f:D(S)\rightarrow\C$
is a \emph{generalized distribution} if it is bounded in the sense
of $\Vert \cdot\Vert^\s_{\infty}$ i.e. there is $K<\o$ such that
$\vert f(x)\vert\le K\Vert x\Vert^\s_{\infty}$.
Then $f$ extends to $C(S)$ by continuity and thus
by generalized distributions we actually mean bounded antilinear
maps $f:C(S)\rightarrow\C$.
\end{definition}

We now fix a generalized distribution $\theta$ for the remainder of this section.

In general a distribution emerges through its test functions, and we will construct test functions picking out rectangles in the spectrum, similarly to how the measure was built. This will let us 'copy over' the values of $\theta$ to our vector. However, we need to be careful to make sure our rectangles are well behaved.

\begin{definition}
We return to the lines
$I_{r}$ and
$J_{r}$ and to the sets
$I_{r}^{\e}$ and
$J_{r}^{\e}$ from Section \ref{sec:2}.
For $X\subseteq\C$, we say that $f\in C(S)$ is an \emph{$X$-function} if
for all $x\in S$, $\vert f(x)\vert\le 1$ and for all $x\in S-X$, $f(x)=0$.
Notice that we do not require that $X\subseteq S$ but still $dom(f)=S$.

We say that a line $I_{r}$ is \emph{good (for $\theta)$} if for all $\d>0$ there is $\e >0$
such that for all $I^{\e}_{r}$-functions $f$,
$\vert \theta (f)\vert <\d$.
$J_{r}$ being good
is defined similarly.
\end{definition}

\begin{lemma}\label{lemma4.3} There are at most countably many $r\in [-M,M]$
such that $I_{r}$ is not good, and the same is true for the lines $J_{r}$.
\end{lemma}

\begin{proof} Suppose not. Then there are $\d>0$ and $r_{i}$, $i<\o_{1}$,
such that for all $\e>0$ and $i<\o_{1}$, there is  an $I^{\e}_{r_{i}}$-function
$f^{\e}_{i}$ such that $\vert\theta (f^{\e}_{i})\vert >\d$.
W.l.o.g. we may assume that for all $i<\o$, $r_{i}<r_{i+1}$.
Then  for all $i<\o$, we choose  $\e_{i}>0$ so that for all
$i<\o$, $r_{i}+\e_{i}<r_{i+1}-\e_{i+1}$.
Also w.l.o.g. we may assume that for all $i<\o$, the direction of
$\l_{i}=\theta (f^{\e_{i}}_{i})$ is roughly the same, e.g.
$Re(\l_{i})\ge Im(\l_{i})\ge 0$.
Let $f=\sum_{i<\o}(1/(i+1))f^{\e_{i}}_{i}$.
Clearly $f\in C(S)$ and $\Vert f\Vert_{\infty}\le 1$
but $\theta (f)=\infty$, a contradiction. \end{proof}

\begin{definition}
We say that $I_{r}$ is \emph{very good (for $\theta)$} if
it is good and 
for all $\d >0$, there is $\e >0$ such that
$\mu^{n}(I^{\e}_{r})<\d$.
$J_{r}$ being very good is defined similarly.
We say that $\e >0$ is \emph{nice for $r$} if
both $r-\e$ and $r+\e$ are nice or $<-M$ or $>M$.
\end{definition}

\begin{lemma}\label{lemma4.4}
\begin{itemize}
\item[(i)] There are at most countably many $r\in [-M,M]$
such that $I_{r}$ is not very good and the same is true for lines $J_{r}$.
\item[(ii)] For all $r$ there  are at most countably many $\e >0$ that
are not nice for $r$.
\item[(iii)] If $\e >0$ is nice for $r$, then $\mu (I_{r}^{\e})=\mu^{n}(I_{r}^{\e})$
and $\mu (J_{r}^{\e})=\mu^{n}(J_{r}^{\e})$.
\end{itemize}
\end{lemma}

\begin{proof} (ii): This is immediate by Lemma \ref{lemma2.1}.

(i): This is immediate by Lemma \ref{lemma2.1} and Lemma \ref{lemma4.3}.

(iii): This can be proved as Lemma  \ref{lemma2.2} (i) was proved.
\end{proof}

\begin{definition}\label{definition4.5} For all  $n<\o$, we choose $r^{n}_{i}\in [-M,M]$, $i<2^{n+2}+1$,
so that
\begin{itemize}
\item[(i)] $r^{n}_{i}<r^{n}_{i+1}$ and $r_{i+1}^{n}-r^{n}_{i}<2M/(n+1)$,
\item[(ii)] $I_{r^{n}_{i}}$ and $J_{r^{n}_{i}}$ are very good for $\theta$,
\item[(iii)] $r^{0}_{0}=-M$ and $r^{0}_{4}=M$,
\item[(iv)] for all $i<2^{n+2}+1$ there is $j<2^{n+3}+1$
such that $r^{n}_{i}=r^{n+1}_{j}$.
\end{itemize}
\end{definition}

When we talk about $I^{\e}_{r^{n}_{i}}$ or $J^{\e}_{r^{n}_{i}}$ we assume that
$\e$ is small enough so that  for all $j<2^{n+2}$,
$r^{n}_{j}+\e <r^{n}_{j+1}-\e$.

For all $n<\o$, $i,j<2^{n+2}$ and $\e >0$, we let
$B^{n}_{ij}=\{ \l\in S\vert\ r^{n}_{i}<Re(\l )<r^{n}_{i+1},\ r^{n}_{j}<Im(\l )<r^{n}_{j+1}\}$
and $R^{n}_{\e}=\cup_{i<2^{n+2}+1}(I^{\e}_{r^n_{i}}\cup  J^{\e}_{r^n_{i}})$.
We say that $\e >0$ is \emph{very nice for $n<\o$}, if it is nice 
for every $r^n_{i}$, $i<2^{n+2}$.
As in Lemma \ref{lemma4.4}, we can see that
$\mu (B^{n}_{ij})=\mu^{n}(B^{n}_{ij})$, that if $\e$ is very nice for $n$,
then $\mu (R^{n}_{\e})=\mu^{n}(R^{n}_{\e})$, and that excluding countably many
$\e$, every $\e$ is very nice for $n$.

\begin{lemma}\label{lemma4.6} For all $n<\o$ and $\d >0$, there is $\e >0$
such that $\mu (R^{n}_{\e})<\d$ and for all $R^{n}_{\e}$-functions $f$,
$\vert\theta (f)\vert <\d$.
\end{lemma}

\begin{proof} For all $R^{n}_{\e}$-functions $f$ there are
$I^{\e}_{r^{n}_{i}}$- and $J^{\e}_{r^{n}_{i}}$-functions $f_{i}$ and $g_{i}$, $i<2^{n+2}$,
such that $f=\sum_{i<2^{n+2}}(f_{i}+g_{i})$. The only nontrivial part here are the 'crossings' between $I$ and $J$ lines, and there one can define $f_i$ to be any continuous continuation from the values it has on the upper and lower edge of the crossing and 0 on the left and right edge, and define $g$ from the differences $f-f_i$. Then the claim follows easily.
\end{proof}

For all  $\l\in\C$, $n<\o$, $i,j<2^{n+2}$ and $\e >0$, we choose a $B^{n}_{i,j}$-function
$f^{n\e\l}_{ij}$ so that for all $x\in\C$, if $r^{n}_{i}+\e<Re(x)<r^{n}_{i+1}-\e$
and  $r^{n}_{j}+\e<Im(x)<r^{n}_{j+1}-\e$, then $f^{n\e\l}_{ij}(x)=\l$.
We write $\theta (B^{n}_{ij},\l )$ for $\lim_{\e\rightarrow 0}\theta (f^{n\e\l}_{ij})$.
Notice that this does not depend on the choice of functions $f^{n\e\l}_{ij}$.
Notice also that $\theta (B^{n}_{ij},\l )=\ol\l \theta(B^{n}_{i,j},1)$.

\begin{definition}\label{definition4.7} Suppose $n:\o\rightarrow\o$ and $\e :\o\rightarrow\R_{+}$.
We say that the pair $(n,\e )$ is \emph{$\theta$-good} if the following
holds:
\begin{itemize}
\item[(i)]  For all $0<m<\o$, $\{ N<\o\vert\ n(N)>m,\ \e(N)<1/m\}\in \U$.
\item[(ii)] There is a $U=U(\theta ,n,\e)\in \U$ such that for all $N\in U$
if $n=n(N)>0$, then
the following holds: Let $\e =\e (N)>0$,
$\d'=\d'(N)=\min\{ \mu_{N}(B^{n}_{ij})\vert\ i,j<2^{n+2},\ \mu (B^{n}_{ij})>0\}$
where
$\mu_{N}(B^{n}_{ij})=\sum_{\l_{N}(k)\in B^{n}_{ij}}\xi_{N}(k)^{2}$,
and let $\d =\d (N)=\d'/(2^{3(n+2)})$. Then
\begin{itemize}
\item[(a)] for all $R^{n}_{\e}$-functions $f$, $\vert\theta (f)\vert <\d$,
\item[(b)] $\sum_{\l_{N}(i)\in R^{n}_{\e}}\xi_{N}(i)^{2}<\d$,
\item[(c)]  for all $i,j <2^{n+2}$, $\vert\mu_{N}(B^{n}_{ij})-\mu (B^{n}_{ij})\vert <\d$.
\end{itemize}
\item[(iii)] $\e (N)$ is very nice for $n(N)$.
\end{itemize}
\end{definition}

Notice that if $(n,\e )$ is $\theta$-good and for all $N<\o$,
$0<\e' (N)<\e (N)$, and $\e'$ is very nice for $n(N)$, then $(n,\e')$ is $\theta$-good. So
keeping in mind that the set of $\e$ very nice for $n$ is dense,
we can always
assume e.g. that $\e (N)$ is such that for all
$i<2^{n(N)+2}$,  $\e (N)<(r^{n(N)}_{i+1}-r^{n(N)}_{i})/3$.

\begin{lemma}\label{lemma4.8} There is a $\theta$-good pair $(n,\e )$ for which we can choose
$U(\theta ,n,\e )=\o$.
\end{lemma}

\begin{proof} Simply let $n(N)$  be the largest natural number $0<n\le N$
for which there is $\e =\e (N)$ for which
Definition \ref{definition4.7} (ii) and (iii) hold, if such  $n$ and  $\e$ exist and otherwise we let
$n(N)=0$ and $\e =1$.
Since for all
$\d >0$, $n<\o$ and $i,j<2^{n+2}$, the set
$$\{ N<\o\vert\ \vert\mu_{N}(B^{n}_{ij})-\mu (B^{n}_{ij}\vert <\d\}$$
belongs to $\U$,
it is easy to see that now also (i) holds.  \end{proof}

\begin{remark} The reader may wonder what the purpose of the sets $U(\theta ,n,\e )$ is.
In the context when we have infinitely many generalized distributions $\theta$
to handle, we need to diagonalize and then  these sets become handy, see Remarks \ref{remark4.11} and \ref{remark5.5}.
\end{remark}

Now we fix a $\theta$-good pair  $(n,\e )$ and in order to simplify
the notations, we assume that  $U(\theta ,n,\e )=\o$.

We are finally ready to define the vector in $H^u$ that will correspond to the distribution $\theta$:
\begin{definition}
We define $u(\theta )=(u_{N}(\theta ))_{N<\o}/\U\in H^{u}$ by letting $n=n(N)$ and defining
$$u_{N}(\theta )=\sum_{\mu (B^{n}_{ij})\ne 0}\sum_{\l_{N}(k)\in B^{n}_{ij}}(\xi_{N}(k)\theta (B^{n}_{ij},1)/\mu_{N}(B^{n}_{ij}))u_{N}(k).$$
\end{definition}

Recall that we may assume the relevant $\mu_N(B^n_{ij})\neq 0$, as this will happen for all $N$ in some $Z\in \U$. Also notice that $u(\theta )$ is $0$-good (immediate by the definition) and that
for all $f\in C(S)$, one can define a generalized distribution
$\theta_{f}$ by $\theta_{f}(g)=\langle g|f\rangle$ and then (if $\theta =\theta_{f}$)
$[F(f)]_{0}=[u(\theta_{f})]_{0}$ (easy calculation).

\begin{lemma}\label{lemma4.9}  $\Vert u(\theta )\Vert_{0}<\infty$.
\end{lemma}

\begin{proof} Let $K$ be such that for all $f\in  C(S)$,
$\vert\theta (f)\vert <K\Vert f\Vert^\s_{\infty}$.
It is enough to show that for all $N<\o$,
$$\sum_{\mu (B^{n}_{ij})\ne 0}\sum_{\l_{N}(k)\in B^{n}_{ij}}\vert\xi_{N}(k)^{2}\theta (B^{n}_{ij},1)/\mu_{N}(B^{n}_{ij})\vert =$$
$$\sum_{\mu (B^{n}_{ij})\ne 0}\vert\theta (B^{n}_{ij},1)\vert\le 16K,$$
where $n=n(N)$.
Suppose not. Then one can find $X\subseteq\{ B^{n}_{ij}\vert\ i,j<2^{n+2},\ \mu (B^{n}_{ij})>0\}$
such that
$$\sum_{B\in X}\vert\theta (B,1)\vert >2K$$
and for all $B\in X$, $\theta (B,1)$ point roughly to the same direction,
e.g. $Re(\theta (B,1))\ge Im(\theta (B,1))\ge 0$, see the proof of Lemma \ref{lemma4.3}.
But now for all $B\in X$, choose a $B$-function $f_{B}$ so that
$\theta (B,1)$ is very close to $\theta (f_{B})$. Let $f=\sum_{B\in X}f_{B}\in C(S)$.
Then $\Vert f\Vert_{\infty}=1$ but $\vert\theta (f)\vert >K$, a contradiction. \end{proof}

\begin{thm}\label{theorem4.10} For all $f\in C(S)$,
$\theta (f)=\langle [F(f)]_{\infty}|[u(\theta )]_{0}\rangle$.
\end{thm}

\begin{proof}  Let $f\in C(S)$.
W.l.o.g. we may assume that
$\vert f(x)\vert\le 1$ for all  $x\in  S$, in particular,
$\Vert f\Vert^\s_{\infty}\le 1$.
For all $0<n<\o$, there is $\e_{1}(n)$ such that if
$x,y\in S$ are such that $\vert x-y\vert <4M/(n(N)+1)$, then
$\vert f(x)-f(y)\vert <\e_{1}(n(N))$ and
$\e_{1}(n)$ goes to $0$ when $n$ goes to infinity.
Then for all $N<\o$ such that $n(N)>0$, we can find
$f_{N}\in C(S)$ such that
$\Vert f-f_{N}\Vert_{\infty}<\e_{1}(n(N))$, $\vert f_{N}(x)\vert\le 1$
for all $x\in S$
and $f_{N}\raj (B^{n(N)}_{ij}-R^{n(N)}_{\e (N)})$ is constant for all $i,j<2^{n+2}$
(here $\e (N)$ is the one from Definition \ref{definition4.7}, not $\e_{1}(n(N))$).
Let $c^{N}_{ij}$ be the constant value.
Now we notice that $\theta (f_{N})$ goes to  $\theta (f)$ when
$n(N)$ goes to infinity  and also
$\langle [F(f_{N})]_{\infty}|[u(\theta )]_{0}\rangle$ goes to
$\langle [F(f)]_{\infty}|[u(\theta )]_{0}\rangle$ when $n(N)$ goes to infinity by
the proof of Lemma \ref{lemma4.2} (and Lemma \ref{lemma4.9}). 
Thus it is enough to show the following: Let $\e^{*} >0$ and
$N<\o$ be such that $n(N)>0$, and
$K/n(N)<\e^{*}$ where $K>0$ is a natural number such that
$\vert\theta (f)\vert\le K\Vert f\Vert^\s_{\infty}$.
Then
$$\vert \theta (f_{N})-\langle F_{N}(f_{N})|u_{N}(\theta)\rangle\vert \le 5\e^{*}.$$
We write $n=n(N)$, $\e=\e(N)$, $\d =\d (N)$ etc.

Now for all $i,j<2^{n+2}$, there is a $B^{n}_{ij}$-function $f_{ij}$
such that for all $x\in B^{n}_{i,j}-R^{n}_{\e}$, $f_{ij}(x)=f_{N}(x)=c^{N}_{ij}$.
Let $f_{2}=\sum_{i,j<2^{n+2}}f_{ij}$. Then (as $\d<\e^*$) it is enough
to show that
$$\vert \theta (f_{2})-\langle F_{N}(f_{N})|u_{N}(\theta)\rangle\vert \le 4\e^{*} .$$
For this it is enough to show that for all $i,j <2^{n+2}$,
$$\vert \theta (f_{ij})-\langle F_{N}(f_{N}\raj B^{n}_{i,j})|u_{N}(\theta)\rangle\vert \le 3K/(2^{3(n+2)}) ,$$
where by $F_{N}(f_{N}\raj B^{n}_{i,j})$ we mean
$\sum_{\l_{N}(k)\in B^{n}_{ij}}\xi_{N}(k)f_{N}(\l_{N}(k))u_{N}(k)$.

Now if $\mu (B^{n}_{ij})=0$, then $B^{n}_{ij}\cap\s (A)=\empty$ by the definition
of $\s^{*}(A)$ and Lemma \ref{lemma2.5} and thus
$\theta (f_{ij})=0$ since
$f_{ij}$ is $\sim_{\infty}$-equivalent with the constant zero function.
Clearly then also $\langle F_{N}(f_{N}\raj B^{n}_{i,j})|u_{N}(\theta)\rangle=0$.
Thus from now on we may assume that $\mu (B^{n}_{ij})>0$.

But then it is enough to show that
$$\vert \theta (f_{ij})-\langle F_{N}(f_{ij})|u_{N}(\theta)\rangle\vert \le 2K/(2^{3(n+2)}),$$
since

\begin{itemize}\item[(*)] if  $\mu (B^{n}_{ij})>0$, then
$$(\sum_{\l_{N}(k)\in B^{n}_{ij}\cap R^{n}_{\e}}\vert\xi_{N}(k)\vert^{2})/\mu_{N}(B^{n}_{ij})\le\d/\d'\le 1/(2^{3(n+2)}).$$
\end{itemize}
\noindent
But then it is enough to show that
$$\vert \theta (B^{n}_{ij},c^{N}_{ij})-\langle F_{N}(f_{ij})|u_{N}(\theta)\rangle\vert \le K/(2^{3(n+2)}) .$$
Since $\theta (B^{n}_{ij}, c^{N}_{ij})=\overline{c^{N}_{ij}}\theta (B^{n}_{ij},1)=\overline{f_{ij}(x)}\theta (B^{n}_{ij},1)$
for any $x\in B^{n}_{ij}-R^{n}_{\e}$,
using $(*)$ above,
this is easy. \end{proof}

\begin{remark}\label{remark4.11} The set of all generalized distributions form a vector space,
we call it GDIS, under
$(\theta +\theta')(x)=\theta (x)+\theta'(x)$ and $q\theta (x)=\theta (\ol q x)$.
Let $V$ be a subspace of this vector space of  countable dimension and
let $\theta_{i}$, $i<\o$, be a basis of this subspace.
Then one can do  the construction for all these generalized distributions  simultaneously:
One can choose the reals $r^{n}_{j}$ so that the lines $I_{r^{n}_{j}}$ and $J_{r^{n}_{j}}$
are very good for every $\theta_{i}$ and then one can find a pair
$(n,\e )$ so that it is $\theta_{i}$-good for every $\theta_{i}$ and in fact a  bit
more: Although we don't get one $U$ set for all the $\theta_i$, we can handle finitely many at a time. In Definition \ref{definition4.7} (ii) (a), one can require
that for all $R^{n}_{\e}$-functions $f$ and $i<n$,
$\vert\theta_{i}(f)\vert <\d/n^{2}$ (see the remark immediately after the definition).
Then for  all $\theta\in V$, one defines $U(\theta ,n,\e )$
to be the set of all $N<\o$ such that
$n(N)>0$ and there are $a_{i}\in\C$, $i<n(N)$,  such that
$\vert a_{i}\vert \le n(N)$ and $\theta  =\sum_{i<n(N)}a_{i}\theta_{i}$.
It is easy to see that by using these $n$  and $\e$, $\theta\mapsto u(\theta)$
is an embedding  of $V$ to  $H^{m}_{0}$ so that Theorem \ref{theorem4.10} holds.
\end{remark}

We can now compare our construction to a rigged Hilbert space. We have defined our distributions using $C(S)$ as test functions. However, our construction disregards what happens outside $\sigma(A)$, and in practice we are looking at $C(\s(A))$, which embeds into $H^m_\infty$. The distributions then embed into $H^m_0$, so we can dig out the rigged Hilbert space in our ultraproduct. However, the construction embedding distributions into $H^m_0$ can only handle countably many distributions at a time (the way we approximate the test functions over rectangles varies with the distributions), asking for a question to be posed:

\begin{question}\label{openquestion4.12}  Can one find an embedding of all of GDIS into $H^{m}_{0}$
so that Theorem \ref{theorem4.10} and everything in Section \ref{sec:4} hold for it?
\end{question}

Another natural question that arises is how canonical the norms $\norm{\cdot}_\infty$ and $\norm{\cdot}_0$ are, as they were constructed from the chosen cyclic vector $\phi$. However, there is no reason to believe that different cyclic vectors would generate the same ultraproduct $H^u$, so different norms act on different spaces. The crucial thing is that we have shown that the vectors corresponding to distributions behave in the desired way (as distributions). Thus, e.g., computing kernels (the way we consider in the next section) in spaces arising from different cyclic vectors will give the same result.

\section{Dirac delta functions and the Feynman propagator}\label{sec:5}

We now turn to the study of the Feynman propagator. In \cite{HH} we studied the kernel vs. Feynman propagator for two simple quantum mechanical systems. To explain the terminology, let us consider a particle in (one-dimensional) space. Its state can be described by a wave function $\varphi$ in position space, where the amplitude $\varphi(x)$ determines the particle's probability of being in position $x$. The evolution over a given time interval is described by the \emph{propagator} $K(x,y)$ which expresses the probability amplitude of the particle travelling form point $x$ to point $y$ in the given time. (Usually the time is given as a third parameter, but with a time independent evolution operator, the question of calculating the kernel can be studied over a fixed time interval.) If the system has eigenvectors $\ket{x}$, $\ket{y}$, corresponding to the positions, one can calculate the propagator as an inner product
$$
K(x,y)=\bra{y}K^{\Delta t}\ket{x}
$$
where $K^{\Delta t}$ is the time evolution operator for the given time interval. In a system without eigenvectors, the propagator is defined as the kernel of the integral representation of of the time evolution operator
$$
K^{\Delta t}(\varphi)(y)=\int_\R K(x,y)\varphi(x)dx.
$$

In this section we show that one can use
Dirac deltas to calculate the kernel of an operator
in the style we tried to calculate it in \cite{HH} (based on
the Feynman propagator). In \cite{HH} the first straightforward approach, to directly calculate the kernel using eigenvectors found in the ultraproduct, failed, essentially because the eigenvectors did not work properly but came in large numbers and had divisibility issues. We remedied the problem by calculating the kernel instead as an average over ever smaller areas.

Here we show how the same averaging idea as in \cite{HH} can be used for the kind of finite dimensional approximations constructed in this paper. Assuming a kernel exists, we show it can be calculated as a limit of approximating inner products. We then show that being very careful one can actually embed the generalized distributions corresponding to Dirac deltas into our ultraproduct model and compute the kernel as a propagator using these Dirac delta vectors. However, the embedding essentially does the same averaging trick in a built in fashion, and is probably not the most convenient way of calculating. The limiting approach we present first is probably much easier to use.

So suppose $B$ is an operator in $H$. Let $G^{-r}:H\rightarrow L_{2}(\s (A),\mu\raj\s (A))$
be the isometric isomorphism determined by the following:
for all $f\in C(S)$,
$G^{-r}(G(f))=f\raj\s (A)$. Let $G^{r}$ be the inverse of $G^{-r}$ and
$B_{D}=G^{-r}BG^{r}$. We suppose that $B$ is such that
there is a continuous function $K(x,y):\s (A)^{2}\rightarrow\C$ such that
for all polynomials $P\in\C [X,Y]$, there is a continuous
$g\in B_{D}(f_{P})$ such that
for all $y\in\s (A)$, $g(y)=\int_{\s (A)}K(x,y)(f_{P}\raj\s (A))(x))dx$.
We will write $B_{D}(f_{P})$ also for this continuous $g$.
Notice that if we are going to calculate the kernel of $B$
in the 'spectral basis of $A$', it
makes sense to assume that $B$ has one.
Also the motivation of these questions comes from physics
and there all functions are continuous and thus we may assume that
$K$ is not only measurable but even continuous.
Notice also that from
this assumption it follows, e.g., that $B$ is bounded
and that for all $f\in D(S)$, $B_{D}(f)\in C(S)$.

However, we want to work in $L_{2}(S,\mu )$ in place of
$L_{2}(\s (A),\mu\raj\s (A))$. For this we choose
a continuation of $K$ to $S^{2}$ and we call this continuation also
$K$ and we call  $G^{-1}BG$ also $B_{D}$ and we notice that
for all polynomials $P\in\C [X,Y]$,
$g(y)=\int_{S}K(x,y)f_{P}(x)dx$ belongs to
$B_{D}(f_{P})$, since
$$\int_{\s (A)}K(x,y)(f_{P}(x)\raj\s (A))dx=(\int_{S}K(x,y)f_{P}(x)dx)\raj\s (A)$$
as one can easily see
and again call also this function
$B_{D}(f_{P})$. If we want to do the calculations  in
the finite dimensional spaces $H_{N}$ we need to find suitable
operators $B_{N}$ on these spaces.

Although we end up working with generalized distributions $\theta$,
the method from \cite{HH} works also inside $H^{m}$ (in fact in $H^{Im}$ as we will see)
which appears a more natural
place to work
and so we start by looking at the $L_{2}$-norm. When we move to
Dirac deltas we will need stronger assumptions. Our first assumption
on the operators $B_{N}$ is the following:

\begin{itemize}
\item[(C1)] There is a natural number $K_{D}$ such that
for all $N<\o$ and $v\in H_{N}$,
$$\Vert B_{N}(v)\Vert <K_{D}\Vert v\Vert .$$
\end{itemize}
Notice that from  this (C1) it follows that
the ultraproduct $B^{u}$ of the operators $B_{N}$ gives a well-defined bounded operator on all of $H^{m}$.
We call this operator $B^{m}$.

The second assumption is the
obvious requirement that if we write
$B^{Im}$ for the restriction of $B^{m}$ to
$H^{Im}$, then

\begin{itemize}
\item[(C2)] $B^{Im}=G^{m}B(G^{m})^{-1}$ (=$F^mB_{D}(F^m)^{-1}$).
\end{itemize}

So how can one find the operators $B_{N}$? First of all, in fact,
they need not be linear functions as long as the ultraproduct of
then is  nice enough. However in practice one probably wants them
to be operators. Examples of  finding these can be found in \cite{HH},
although the situation there is not exactly the same
as  here. In the case of the free particle,
$B$ had a definition in terms of an operator $C$ for which
we already had operators $C_{N}$ and we used this definition
in the spaces $H_{N}$ to get operators $B_{N}$. In the case of the harmonic oscillator
we could have done essentially the same.  However
it turned out that with this definition $K(x,y)$ was very difficult to
calculate. Thus we used another method that gave
completely different operators $B_{N}$ but  whose ultraproduct
was the same (upto $\sim_{2}$) in the places that mattered.

In the next lemma we show that the operators $B_{N}$
can always be found. The proof is existential i.e. the method used there
can not be used to find the operators in practice for an obvious reason
- unless the mere existence of them is enough.
And as pointed out above, the best way of choosing the operators
may be such that it gives operators that
are very different from those chosen in the proof of the 
lemma. The operators defined in the proof will be useful later.

\begin{lemma}\label{lemma5.1} There are operators $B_{N}$ in the spaces  $H_{N}$
such that they satisfy (C1) and (C2) above.
\end{lemma}

\begin{proof} We define $B_{N}$ as follows.
Let $v=\sum_{n=0}^{D_{N}-1}a_{n}u_{N}(n)\in  H_{N}$. Then we let
$$B_{N}(v)=\sum_{n=0}^{D_{N}-1}\sum_{k=0}^{D_{N}-1}\xi_{N}(k)\xi_{N}(n)a_{k}K(\l_{N}(k),\l_{N}(n))u_{N}(n).$$
Clearly  $B_{N}$ is a linear function.

Let $K_{D}$ be a natural number greater than any of the absolute values of the values of the
function $K$. Then
$$\Vert B_{N}(v)\Vert^2_{2}\le\sum_{n=0}^{D_{N}-1}\sum_{k=0}^{D_{N}-1}\xi_{N}(k)^{2}\xi_{N}(n)^{2}\vert  a_{k}K(\l_{N}(k),\l_{N}(n))\vert^{2}\le$$
$$\le K_{D}^{2}\Vert v\Vert^2_{2}\sum_{n=0}^{D_{N}-1}\xi_{N}(n)^{2}=K_{D}^{2}\Vert v\Vert^2_{2}.$$
So (C1) holds.

From the definition of the functions $B_{N}$
and the uniform continuity of the functions $K$ and $f_{P}$, it is easy to see
that
for all $P\in\C [X,Y]$, $[(B_{N}(F_{N}(f_{P})))_{N<\o}]_{2}=[(F_{N}(B_{D}(f_{P})))_{N<\o}]_{2}$.
But then (C2) follows (since the vectors  $[(F_{N}(f_{P}))_{N<\o}]_{2}$,
$P\in\C [X,Y]$, are dense in
$H^{Im}$).
\end{proof}

From now on we will write $\tilde{B}_{N}$ for the operators $B_{N}$ from the proof
of Lemma \ref{lemma5.1} $\tilde{B}^{u}$ for their ultraproduct
and $\tilde{B}^{m}$ for their metric ultraproduct.

Now we want to calculate $K(\a ,\b )$ for $\a ,\b\in\s (A)$ from operators $B_{N}$, $N<\o$,
that satisfy (C1) and (C2).
We start by looking at the method used in \cite{HH}, namely computing the kernel as a limit of kernels (or propagators, as computed in the finite dimensional spaces). Then we get the Dirac
delta method as an immediate consequence. Now we choose numbers $r^{n}_{i}$ for $n<\o$ and $i<2^{n+2}+1$
(and lines $I_{r^{n}_{i}}$ and $J_{r^{n}_{i}}$) as in Section \ref{sec:4} so that
Definition \ref{definition4.5} holds when (ii) is replaced by

(ii)' for all $m<\o$, neither $\a$ nor $\b$ is in any line $I_{r^{m}_{i}}$ or $J_{r^{m}_{i}}$, $i<s^{m+2}+1$, and
for all $\d>0$ there is $\e >0$ such that $\mu^{n}(R^{m}_{\e})<\d$,
where $R^{m}_{\e}$ is as in Section \ref{sec:4}.

For all $p<\o$,  we let $B^p_\a$ be the $B^p_{ij}$ containing $\a$, and let $u^{p}_{\a}=(u^{p}_{\a}(N))_{N<\o}$, where
$$u^{p}_{\a}(N)=\sum_{\l_{N}(k)\in B^{p}_{\a}}(\xi_{N}(k)/\mu_{N}(B^{p}_{\a}))u_{N}(k)$$
and similarly for $\b$ in place of $\a$. Notice that
$\Vert u^{p}_{\a}\Vert^2_{2}=1/\mu (B^{p}_{\a})$ and thus
$[u^{p}_{\a}]_{2}\in H^{m}_{2}$ whenever $\a\in\s(A)$.

\begin{thm}\label{theorem5.2} When $\a,\b\in\s(A)$, $K(\a ,\b )=\lim_{p\rightarrow\infty}\langle [u^{p}_{\b}]_{2}\vert B^{m}([u^{p}_{\a}]_{2})\rangle$.
\end{thm}

\begin{proof} We write $\chi_{B^{n}_{ij}}$ for the characteristic function of $B^{n}_{ij}$ and
$\chi'_{B^{n}_{ij}}$ for the function
$\chi'_{B^{n}_{ij}}(x)=\chi_{B^{n}_{ij}}(x)/\mu (B^{n}_{ij})$.
Since we can approximate $\chi'_{B^{p}_{\a}}$ by polynomials in the $L_{2}$-norm,
it is easy to see that
$$[F(\chi'_{B^{p}_{\a}})]_{2}=[u^{p}_{\a}]_{2}.$$
Notice that from this it follows that
$[u^{p}_{\a}]_{2}\in H^{Im}$.
Now using this, the approximations again, (C1), (C2) and the proof
of Lemma \ref{lemma5.1},
$$B^{m}([u^{p}_{\a}]_{2})=\tilde{B}^{m}([u^{p}_{\a}]_{2})=[\tilde{B}^{u}(u^{p}_{\a})]_{2}.$$ 

We are left with an easy calculation:
\begin{align*}
\langle u^{p}_{\b}(N)\vert &\tilde{B}_{N}(u^{p}_{\a}(N))\rangle=\\
=\langle \sum_{\l_{N}(k)\in B^{p}_{\b}}(\xi_{N}(k)/\mu_{N}(B^{p}_{\b}))u_{N}(k)\vert& \\
\sum_{n=0}^{D_{N}-1}\sum_{\l_{N}(l)\in B^{p}_{\a}}
K(\l_{N}(l),\l_{N}(n))& \xi_{N}(l)^{2}\xi_{N}(n)\mu_{N}(B^{p}_{\a})^{-1}u_{N}(n)\rangle =\\
=\sum_{\l_{N}(k)\in B^{p}_{\b}}\sum_{\l_{N}(l)\in B^{p}_{\a}}
K(\l_{N}(l),\l_{N}(k))&\xi_{N}(l)^{2}\xi_{N}(k)^{2}(\mu_{N}(B^{p}_{\a})\mu_{N}(B^{p}_{\b}))^{-1}.
\end{align*}
Now  keeping in mind that $K$ is a continuous function, when $p$ is large enough,
this is roughly
$$\sum_{\l_{N}(k)\in B^{p}_{\b}}\sum_{\l_{N}(l)\in B^{p}_{\a}}K(\a,\b)\xi_{N}(l)^{2}\xi_{N}(k)^{2}(\mu_{N}(B^{p}_{\a})\mu_{N}(B^{p}_{\b}))^{-1}=K(\a ,\b),$$
since the absolute value of the error is at most
the maximum of
$$\vert K(\a ,\b )-K(\l_{N}(l),\l_{N}(k))\vert ,$$
for $\l_{N}(k)\in B^{p}_{\b}$ and $\l_{N}(l)\in  B^{p}_{\a}$
as a straight forward calculation shows. \end{proof}

Now we can turn to Dirac deltas. We let $\theta_{\a}$ be the generalized distribution such that
$\theta_{\a}(f)=f(\a )$ for all $f\in C(S)$ and $\theta_{\b}$
is defined similarly. Now it is easy to see that for all
$n<\o$ and $i<2^{n+2}+1$, $I_{r^{n}_{i}}$ and $J_{r^{n}_{i}}$ are very good for
both $\theta_{\a}$ and $\theta_{\b}$ i.e. all the requirements of Definition \ref{definition4.5}
are satisfied. It follows that we can find a pair $(n,\e )$
so that it is  $\theta_{\a}$-good and $\theta_{\b}$-good  (see the end of Section \ref{sec:4}).
We will need one more requirement for the pair $(n,\e )$. We will return to this
once we have looked at the requirements for the operators $B_{N}$, $N<\o$.

Again, recall that the pair $(n,\e )$ determines the function $u$
and (independent of the last requirement)
notice that
$\Vert u(\theta_{\a})\Vert^{u}_{0}=1$
and thus 
$[u(\theta_{\a})]_{0}\in H^{m}_{0}$
(and similarly for $\b$). However $\Vert u^{p}_{\a}\Vert_{\infty}$ is infinite
and this causes problems, the inner product in the Feynman propagator may not be well-defined.
So we need to modify the assumptions (C1) and (C2).

We let $H^{\infty}_{\infty}$  be $H^{u}/\sim_{\infty}$. This is a natural space to consider Dirac deltas, as it is large enough to allow for 'infinite' objects, and on the other hand has a very delicate equivalence class structure, avoiding the risk of putting into the same equivalence class objects that behave differently in our calculations.
Now our first assumption requires that the ultraproduct of
the operators $B_{N}$ gives a well-defined operator on $H^{\infty}_{\infty}$.
We will call the operator $B^{\infty}$:

\begin{itemize}\item[(C1)'] There is a natural number $K_{D}$ such that
for all $N<\o$ and $v\in H_{N}$,
$$\Vert B_{N}(v)\Vert^{S}_{\infty} <K_{D}\Vert v\Vert^{S}_{\infty}.$$
\end{itemize}

Our second assumption ties the values of $B_{N}$ to  values of
$B_{D}$ in the sense of  $\Vert \cdot\Vert_{\infty}$. Here we use
basically the simple functions since our definition
of the function $u$ is based on them. However, notice that
instead of simple functions 
we could use polynomials here as well as in Section \ref{sec:4},
but this would make the definition of $u$ much more complicated.

We write $F(\chi'_{B^{n}_{ij}})$ for $(F_{N}(\chi'_{B^{n}_{ij}}))_{N<\o}$ where
$$F_{N}(\chi'_{B^{n}_{ij}})=\sum_{n=0}^{D_{N}-1}\xi_{N}(n)\chi'_{B^{n}_{ij}}(\l_{N}(n))u_N(n).$$
We also write $B_{D}(\chi'_{B^{n}_{ij}})$ for the
function $g(y)=\int_{S}K(x,y)\chi'_{B^{n}_{ij}}(x)dx$.
Then (C2) implies 
$$B^{m}([F(\chi_{B^{n}_{ij}})]_{2})=[F(B_{D}(\chi_{B^{n}_{ij}})]_{2}.$$
And now following this we assume:

\begin{itemize}\item[(C2)'] For all $B^{n}_{ij}$,
$$B^{\infty}([F(\chi'_{B^{n}_{ij}})]_{\infty})=[F(B_{D}(\chi'_{B^{n}_{ij}}))]_{\infty}.$$
\end{itemize}

\noindent
We recall that if $\Vert(v_{N})_{N<\o}/\U\Vert_{\infty}<\infty$ and for all $N<\o$,
$v_{N}=\sum_{n=0}^{D_{N}-1}a^{N}_{n}u_{N}(n)$, then there are $Z\in \U$ and
a natural number $p$ such that for all $N\in Z$ and $n<D_{N}$,
$\vert\xi_{N}(n)^{-1}a^{N}_{n}\vert <p$. However,
$B^{\infty}$ images of Dirac's deltas diagonalize such functions
and thus we may loose this property. So although with (C1)'  and (C2)'
and (*) below, $u(\theta_{\a})$ and and $u(\theta_{\b})$ do give Feynman's
propagator the value $K(\a ,\b )$, this may be kind of accidental,
they have just got lucky. The value is right for wrong reasons.
So we introduce one more requirement for the operators $B_{N}$
(it is a bit unnecessarily strong, but we are just making a point).

\begin{itemize}\item[(C3)'] There is a natural number $p$ such that for all $n<\o$,
$i,j<2^{n+2}+1$ and $N<\o$, $\Vert B_{N}(u^{n}_{ij}(N))\Vert^{S}_{\infty}<p$,
where
$$u^{n}_{ij}(N)=\sum_{\l_{N}(k)\in B^{n}_{ij}}(\xi_{N}(k)/\mu_{N}(B^{n}_{ij}))u_{N}(k)$$
when $\mu_N(B^n_{ij})\neq 0$, and 0 otherwise.
\end{itemize}

\begin{lemma}\label{lemma5.3} The operators $\tilde{B}_{N}$, $N<\o$, satisfy (C1)', (C2)' and (C3)'.
\end{lemma}

\begin{proof} (C1)' and (C2)' can be proved essentially as in the proof of Lemma \ref{lemma5.1} and
(C3)' is a straight forward calculation.
\end{proof}

Let us return to the choice of the pair $(n,\e )$ and thus to the choice of $u$.
We choose it so that in addition to what we have already said,
the following holds: There are $U(\theta_{\a},n,\e )$ and $U(\theta_{\b},n,\e )$
that witness that $(n,\e )$ is both $\theta_{\a}$-good and $\theta_{\b}$-good and

\begin{itemize}\item[(*)] if $N\in U(\theta_{\a},n,\e )\cap U(\theta_{\b},n,\e )$ and
$p=n(N)>0$, then
$$\vert \langle u^{p}_{\b}(N)\vert B_{N}(u^{p}_{\a}(N))\rangle-\langle u^{p}_{\b}\vert B^{u}(u^{p}_{\a})\rangle\vert <p^{-1}.$$
\end{itemize}

\noindent
It is easy to see that this is possible.

\begin{proposition}\label{proposition5.4} Suppose (C1)', (C2)' and (C3)' hold, and $\a,\b\in\s(A)$. Then

$$K(\a ,\b )=\langle [u(\theta_{\b})]_{0}| B^{\infty}([u(\theta_{\a})]_{\infty})\rangle.$$
\end{proposition}

\begin{proof} Notice that if $p=n(N)>0$, then
  $F_N(\chi_{B^p_\a}\mu_N(B^p_\a)^{-1})=u^{p}_{\a}(N)=u_{N}(\theta_{\a})$
  and
$(F_N(\chi_{B^p_\a}\mu_N(B^p_\a)^{-1}))_{N<\o}/\U\in[F(\chi'_{B^p_\a})]_\infty$.
  and so
from (C3)' it follows that $B^{\infty}[u(\theta_{\a})]_{\infty}\in H^{m}_{\infty}$
and thus by  Lemma \ref{lemma4.2},
$$\langle [u(\theta_{\b})]_{0}| B^{\infty}([u(\theta_{\a})]_{\infty})\rangle=\overline{\langle B^{\infty}([u(\theta_{\a})]_{\infty})|[u(\theta_{\b})]_{0}\rangle}<\infty ,$$
i.e. it is well-defined.
By (C2)' and Lemma \ref{lemma5.3}, $B^{\infty}([u^{n}_{\a}]_{\infty})=[\tilde{B}^{u}(u^{n}_{\a})]_{\infty}$, for all $n<\o$.
But then the claim follows from (the proof of) Theorem \ref{theorem5.2} and (*) above.
\end{proof}

\begin{remark}\label{remark5.5} Let $Z$ be the set of all $r$ for which
there is $\d >0$ such that either for all $\e >0$
$\mu^{n}(I_{r}^{\e})\ge\d$  or for all $\e >0$, $\mu^{n}(J^{\e}_{r})\ge\d$.
In Section \ref{sec:2} we saw that  $Z$ is countable.
Let
$Y=\cup_{r\in Z}(I_{r}\cup J_{r})$. Then
as in the end of Section \ref{sec:4}, we can see that for any countable
$X\subseteq\s (A)-Y$,
we can choose a pair $(n,\e )$  so that
it is $\theta_{\a}$-good for all $\a\in X$ i.e. Theorem \ref{theorem4.10} holds, (*) above holds
for every pair $(\a,\b )\in X^{2}$ and in addition $u$
gives an embedding of the vector space generated by the Dirac deltas $\theta_{\a}$,
$\a\in X$, into $H^{m}_{0}$.
So we can make everything in Sections \ref{sec:3} and \ref{sec:4}
work for all $\theta_{\a}$, $\a\in X$, simultaneously.
If $X$  is dense in $\s (A)$, one expects  that this is enough for most arguments
in the contexts in which all functions are continuous.
\end{remark}

\section{The non-cyclic case}\label{sec:non-cyclic}

We finally look at the case where the operator doesn't have a cyclic vector. Then $H$ can be decomposed into a countable direct sum $\bigoplus_{i<\omega} H_i$ of subspaces $H_i$, invariant under $A$ and with a cyclic vector $v_i$. By the construction for the cyclic case, each $H_i$ is isomorphic to a space $L_2(S,\mu_i)$, where every $\mu_i$ is a Borel measure with $\mu_i(S)=1$.

Now, to combine these, let
$$
\mu=\sum_i 2^{-i-1}\mu_i.
$$
Then $\mu$ is a finite Borel measure, and every $\mu_i$ is absolutely continuous with respect to $\mu$. Thus by the Radon-Nikodym theorem there are Borel functions $f_i$ such that for all Borel sets $A\subseteq S$
$$
\mu_i(A)=\int_Af_id\mu.
$$
This means that each $L_2(S,\mu_i)$ can be isometrically embedded into $L_2(S,\mu)$ by the mapping generated by 
$$
g(x)\mapsto \sqrt{f_i(x)}g(x).
$$
Now if we wish to combine the direct sum of these into an $L_2$ space over a finite measure space, we can turn to the \emph{Bochner spaces}, which are function spaces of vector valued functions (see, e.g., \cite{Hy}). We can recognize the space $\oplus_{i<\omega}L_2(S,\mu)$ as a Bochner-$\ell_2$-space with values in the space $L_2(S,\mu;\C)$, i.e., $\ell_2(L_2(S,\mu;\C))$. By \cite[Corollary 1.2.23]{Hy} $\ell_2(L_2(S,\mu;\C))\cong L_2(S,\mu;\ell_2(\C))$, where the latter is the vector valued $L_2$-space over $S$ with values in $\ell_2(\C)$.
Our sum space $\oplus_iH_i$ can now be isometrically embedded into this space by combining the isometric isomorphisms $H_i\cong L_2(S,\mu_i)$ obtained from our previous theory, the embeddings $L_2(S,\mu_i)\hookrightarrow L_2(S,\mu)$, and finally the isomorphism $\ell_2(L_2(S,\mu;\C))\cong L_2(S,\mu;\ell_2(\C))$.

If we denote the image under this embedding by $H'$, we can as before find a copy of $H'$ within a suitable ultraproduct of finite-dimensional spaces: If we denote by $H^i_N$ the finite dimensional subspaces of $H_i$ constructed as before, we can construct spaces
$$
H^N=\oplus_{i\leq N}H^i_N,
$$
and build the ultraproduct model from these. A canonical modification of the embedding $G^m$ gives an embedding of $\oplus_{i<\omega}H_i$ (and thus of $H'$) into the metric ultraproduct and allows us to use finite dimensional approximations. It is worth noting that this does not give an embedding of all of $L_2(S,\mu;\ell_2(\C))$ into the metric ultraproduct.

Now, although we can find a spectral decomposition of our space in an ultraproduct of finite-dimensional spaces, we don't see that this would give a general approach to kernels. If the operator $B$ behaves differently in the different spaces $H_i$, one doesn't get a common complex-valued kernel that could be used for the vectors of the Bochner space. And if the $H_i$ aren't invariant under $B$, one cannot treat these spaces separately, either.

\end{document}